\documentclass[10pt]{amsart}
\usepackage{amsmath,amsthm, epsfig, amscd}
\usepackage{psfrag}
\usepackage[psamsfonts]{amssymb}
\usepackage[all]{xy}
\usepackage{mathrsfs}

\numberwithin{equation}{section}

\newtheorem{Thm}{Theorem}[section]
\newtheorem{Prop}[Thm]{Proposition}
\newtheorem{Lem}[Thm]{Lemma}
\newtheorem{Cor}[Thm]{Corollary}

\theoremstyle{remark}
\newtheorem{Rem}[Thm]{Remark}

\theoremstyle{definition}
\newtheorem{Def}[Thm]{Definition}

\newtheorem{Exa}[Thm]{Example}

\newcommand{\mysection}[2]{%
\vspace{2mm}\section{\bf #1}\label{#2}
}

\newcommand{\fig}[1]
        {\raisebox{-0.5\height}
                 {\includegraphics{#1}}
        }

\def\Z{{\mathbb Z}}
\def\R{{\mathbb R}}
\def\Q{{\mathbb Q}}

\def\calA{\mathscr{A}}

\def\calD{\mathscr{D}}

\def\calM{\mathscr{M}}

\def\deg{\mathrm{deg}}

\newcommand{\mapright}[1]{
	\smash{\mathop{
		\hbox to 1cm{\rightarrowfill}}\limits^{#1}}}

\newcommand{\mapleft}[1]{
	\smash{\mathop{
		\hbox to 1cm{\leftarrowfill}}\limits^{#1}}}

\def\bcalM{\overline{\calM}}

\def\wgamma{\widetilde{\gamma}}

\def\ve{\varepsilon}
\def\bConf{\overline{C}}

\def\acalM{\calM^{\mathrm{AL}}}
\def\bacalM{\bcalM^{\mathrm{AL}}}

\def\wW{\widetilde{W}}
\def\wL{\widetilde{L}}
\def\wM{\widetilde{M}}
\def\wf{\widetilde{f}}
\def\wGamma{\widetilde{\Gamma}}
\def\bS{\overline{S}}
\def\bT{\overline{T}}
\def\bG{\overline{G}}
\def\bH{\overline{H}}
\def\wcalD{\widetilde{\calD}}
\def\wcalA{\widetilde{\calA}}

\def\twedge{\textstyle\bigwedge}
\def\irr{\mathrm{irr}}

\begin{document}

\title[Morse theory and Lescop's equivariant propagator]{Morse theory and Lescop's equivariant propagator for 3-manifolds with $b_1=1$ fibered over $S^1$}
\author{Tadayuki Watanabe}
\address{Department of Mathematics, Shimane University,
1060 Nishikawatsu-cho, Matsue-shi, Shimane 690-8504, Japan}
\email{tadayuki@riko.shimane-u.ac.jp}
\date{\today}

{\noindent\footnotesize {\rm Preprint} (2014)}\par\vspace{15mm}
\maketitle
\vspace{-6mm}
\setcounter{tocdepth}{1}
\begin{abstract}
For a 3-manifold $M$ with $b_1(M)=1$ fibered over $S^1$ and the fiberwise gradient $\xi$ of a fiberwise Morse function on $M$, we introduce the notion of amidakuji-like path (AL-path) in $M$. An AL-path is a piecewise smooth path on $M$ consisting of edges each of which is either a part of a critical locus of $\xi$ or a flow line of $-\xi$. Counting closed AL-paths with signs gives the Lefschetz zeta function of $M$. The ``moduli space" of AL-paths on $M$ gives explicitly Lescop's equivariant propagator, which can be used to define $\mathbb{Z}$-equivariant version of Chern--Simons perturbation theory for $M$.
\end{abstract}
\par\vspace{3mm}

\def\baselinestretch{1.07}\small\normalsize


\mysection{Introduction}{s:intro}

Chern--Simons perturbation theory for 3-manifolds was developed independently by Axelrod--Singer (\cite{AS}) and by Kontsevich (\cite{Ko}). It is defined by integrations over suitably compactified configuration spaces $\bConf_{2n,\infty}(M)$ of a closed 3-manifold $M$ and gives a strong invariant $Z(M)$ of $M$ whose universal formula is a formal series of Feynman diagrams (e.g. \cite{KT, Les1}). In the definition of $Z$, propagator plays an important role. Here, a propagator is a certain closed 2-form on $\bConf_{2,\infty}(M)$, which corresponds to an edge in a Feynman diagram (see \cite{AS, Ko} for the definition of propagator, and \cite{Les1} for a detailed exposition). The Poincar\'{e}--Lefschetz dual to a propagator is given by a relative 4-cycle in $(\bConf_{2,\infty}(M),\partial \bConf_{2,\infty}(M))$. In a dual perspective, given three parallels $P_1,P_2,P_3$ of such a 4-cycle, the algebraic triple intersection number $\#P_1\cap P_2\cap P_3$ in the 6-manifold $\bConf_{2,\infty}(M)$ gives rise to the 2-loop part of $Z$, which corresponds to the $\Theta$-shaped Feynman diagram. For general 3-valent graphs with $2n$ vertices, the intersections of codimension 2 cycles in $\bConf_{2n,\infty}(M)$ give rise to invariants of $M$.

Propagator may not exist depending on the topology of $M$. For a propagator with $\Q$ coefficients to exist, $M$ must be a $\Q$ homology 3-sphere. If $M$ is a closed 3-manifold with $b_1(M)>0$, one must improve the method to find universal perturbative invariant for $M$ whose value is a formal series of Feynman diagrams, which is not classical. After Ohtsuki's pioneering work that refines the LMO invariant significantly (\cite{Oh1, Oh2}), Lescop gave a topological construction of an invariant of $M$ with $b_1(M)=1$ for the 2-loop graph using configuration spaces. More precisely, she defined in \cite{Les2} a topological invariant of $M$ by using the equivariant triple intersection of ``equivariant propagators'' in the ``equivariant configuration space'' $\bConf_2(\wM)_\Z$ of $M$. The equivariant configuration space $\bConf_2(\wM)_\Z$ is an infinite cyclic covering of the compactified configuration space $\bConf_2(M)$. An equivariant propagator is defined as a relative 4-cycle in $(\bConf_2(\wM)_\Z,\partial\bConf_2(\wM)_\Z)$ with coefficients in $\Q(t)$ satisfying a certain boundary condition, which is described by a rational function including the Alexander polynomial of $M$ (\cite[Theorem~4.8]{Les2} or Theorem~\ref{thm:les2} below). She proved that some equivalence class of the equivariant triple intersection of three equivariant propagators gives rise to an invariant of $M$. Note that the existence of an equivariant propagator satisfying an explicit boundary condition is proved in \cite{Les2}, whereas globally explicit cycle is not referred to except for the case $M=S^2\times S^1$. 

In this paper, we introduce the notion of amidakuji-like path (AL-path for short) in a 3-manifold $M$ with $b_1(M)=1$ fibered over $S^1$ (Definition~\ref{def:al-path}) and we construct Lescop's equivariant propagator explicitly as the chain given by the moduli space of AL-paths in $M$ (Theorem~\ref{thm:main}). In proving the main Theorem~\ref{thm:main}, we show that the counts of closed AL-paths in $M$ give the Lefschetz zeta function of the fibration $M$ (Proposition~\ref{prop:gamma_zeta}). In a sense, our construction gives a geometric derivation of the formula for Lescop's boundary condition.

AL-path is in a sense a piecewise smooth approximation of integral curve of a nonsingular vector field on $M$ (see Figure~\ref{fig:converge_curve}). Let $\xi$ be the gradient along the fibers of a fiberwise Morse function (\S\ref{ss:fMF}) of $M$. Roughly speaking, an AL-path in $M$ is a piecewise smooth path in $M$ that is an alternating concatenation of horizontal segments and vertical segments, where a horizontal segment is a part of a flow line of $-\xi$ and a vertical segment is a part of a critical locus of $\xi$ both descending. 

Explicit propagator is good for finding explicitly computable invariant of 3-manifolds. Inspired by the ideas of \cite{Fu, Wa1} for construction of graph-counting invariants for homology 3-spheres, we obtain a candidate for equivariant version of the Chern--Simons perturbation theory for 3-manifolds $M$ with $b_1(M)=1$ fibered over $S^1$, by counting graphs in $M$ each of whose edges is an AL-path for a fiberwise gradient. We will write about it in \cite{Wa2}. We believe that our construction can be extended to 3-manifolds with arbitrary first Betti numbers and to generic closed 1-forms, generic in the sense of \cite{Hu}, by using a method similar to that of Pajitnov in \cite{Pa1, Pa2}. 

\subsection{Conventions}
In this paper, manifolds and maps between them are assumed to be smooth. By an $n$-dimensional {\it chain} in a manifold $X$, we mean a finite linear combination of smooth maps from oriented compact $n$-manifolds with corners to $X$. We understand a chain as a chain of smooth simplices by taking triangulations of manifolds. We follow \cite[Appendix]{BT} for the conventions for manifolds with corners and fiber products of manifolds with corners. Some definitions needed are summarized in Appendix~\ref{s:mfd_corners}. We represent an orientation $o(X)$ of a manifold $X$ by a non-vanishing section of $\twedge^{\dim{X}} T^*X$. We consider a coorientation $o^*(V)$ of a submanifold $V$ of a manifold $X$ as an orientation of the normal bundle of $V$ and represent it by a differential form in $\Gamma^\infty(\twedge^\bullet T^*X|_{V})$. We identify the normal bundle $N_V$ with the orthogonal complement $TV^\perp$ in $TX$, by taking a Riemannian metric on $X$. We fix orientation or coorientation of $V$ so that the identity
\[ o(V)\wedge o^*(V)\sim o(X) \]
holds, where we say that two orientations $o$ and $o'$ are equivalent ($o\sim o'$) if they are related by multiple of a positive function. $o(V)$ determines $o^*(V)$ up to equivalence and vice versa. We orient boundaries of an oriented manifold by the inward normal first convention.

\subsection{Lefschetz zeta function}

We shall recall a few definitions and notations before stating the main result. Let $\Sigma$ be a closed manifold. For a diffeomorphism $\varphi:\Sigma\to \Sigma$, its {\it Lefschetz zeta function} $\zeta_\varphi(t)$ is defined by the formula
\[ \zeta_\varphi(t)=\exp\left(\sum_{k=0}^\infty \frac{L(\varphi^k)}{k}t^k\right)\in\Q[[t]], \]
where $L(\varphi^k)$ is the Lefschetz number of the iteration $\varphi^k$, or the count of fixed points of $\varphi^k$ counted with appropriate signs. The following product formula is a consequence of the Lefschetz trace formula.
\begin{equation}\label{eq:zeta_prod}
 \zeta_\varphi(t)=\prod_{i=0}^{\dim{\Sigma}}\det(1-t\varphi_{*i})^{(-1)^{i+1}}, 
\end{equation}
where $\varphi_{*i}:H_i(\Sigma;\Q)\to H_i(\Sigma;\Q)$ is the induced map from $\varphi$. See e.g. \cite[9.2.1]{Pa2}. In this paper, we will often consider the logarithmic derivative of $\zeta_\varphi(t)$. One has
\begin{equation}\label{eq:dlog_zeta}
 \frac{d}{dt}\log{\zeta_\varphi(t)}=\frac{\zeta_\varphi'(t)}{\zeta_\varphi(t)}
=\sum_{i=0}^{\dim{\Sigma}}(-1)^i\mathrm{Tr}\frac{\varphi_{*i}}{1-t\varphi_{*i}}.
\end{equation}

\subsection{Equivariant configuration spaces}

We recall some definitions from \cite{Les2}. Let $M$ be a closed oriented Riemannian 3-manifold with $b_1(M)=1$, let $\kappa:M\to S^1$ be a map that induces an isomorphism $H_1(M)/\mathrm{Torsion}\to H_1(S^1)$ and let $\wM$ be its standard infinite cyclic covering. Let $\pi:\wM\to M$ be the covering projection. Let $\widetilde{\kappa}:\wM\to \R$ be the lift of $\kappa$ and let $t:\wM\to \wM$ be the diffeomorphism that generate the group of covering transformations and that satisfies for every $x\in\wM$,
\[ \widetilde{\kappa}(tx)=\widetilde{\kappa}(x)-1. \]
Let $\wM\times_\Z\wM$ be the quotient of $\wM\times \wM$ by the equivalence relation that identifies $x\times y$ with $tx\times ty$. We denote the equivalence class of $x\times y$ by $x\times_\Z y$. The natural map $\bar{\pi}:\wM\times_\Z\wM\to M\times M$ is an infinite cyclic covering. By abuse of notation, we denote by $t$ the generator of the group of covering transformations of $\wM\times_\Z\wM$ that acts as follows.
\[ t(x\times_\Z y)=(t^{-1}x)\times_\Z y=x\times_\Z (ty). \]

Let $\Delta_M$ be the diagonal in $M\times M$. The compactified configuration space $\bConf_2(M)$ is the compactification of $M\times M\setminus\Delta_M$ that is obtained from $M\times M$ by blowing-up $\Delta_M$. See Appendix~\ref{s:blow-up} for the definition of blow-up. Roughly, the blow-up replaces $\Delta_M$ with its normal sphere bundle. The boundary $\partial\bConf_2(M)$ is canonically identified with the unit tangent bundle $ST(M)$ of $M$. More precisely, let $N_{\Delta_M}$ be the total space of the normal bundle of $\Delta_M$ in $M\times M$. We fix a framing $\tau:TM\to \R^3\times M$. The framing of $M$ induces an isomorphism
\begin{equation}\label{eq:triv_normal}
 \phi:N_{\Delta_M}\to \R^3 \times \Delta_M
\end{equation}
of oriented vector bundles, which sends the fiber $(T_{(x,x)}\Delta_M)^\perp$ of the normal bundle to $\R^3\times (x,x)$. Then $\phi$ induces a diffeomorphism $B\ell_0(N_{\Delta_M})\to B\ell_0(\R^3)\times \Delta_M$. Under this diffeomorphism, the boundary of $B\ell_0(N_{\Delta_M})$ corresponds to $\partial B\ell_0(\R^3)\times \Delta_M=S^2\times \Delta_M\approx S^2\times M$. We denote $\phi^{-1}(S^2\times \Delta_M)$ by $ST(M)$. Note that the blowing-up does not depend on the choice of $\tau$.

Let $\widetilde{\Delta}_M=\bar{\pi}^{-1}(\Delta_M)$. The {\it equivariant configuration space} $\bConf_2(\wM)_\Z$ is defined by
\[ \bConf_2(\wM)_\Z=B\ell_{\widetilde{\Delta}_M}(\wM\times_\Z\wM),\]
the blow-up of $\wM\times_\Z\wM$ along $\widetilde{\Delta}_M$. The boundary of $\bConf_2(\wM)_\Z$ is canonically identified with $\Z\times ST(M)=\coprod_{i\in\Z}t^iST(M)$.

\subsection{Lescop's equivariant propagator}

Let $K$ be an oriented knot in $M$ such that $\langle[d\kappa], [K]\rangle=-1$. Let $\Lambda=\Q[t,t^{-1}]$ and let $\Q(t)$ be the field of fractions of $\Lambda$. Then $H_*(\bConf_2(\wM)_\Z)$ is naturally a graded $\Lambda$-module. 
\begin{Thm}[Lescop \cite{Les2}] For any $i\in\Z$, 
\[ H_i(\bConf_2(\wM)_\Z)\otimes_\Lambda \Q(t)\cong H_{i-2}(M;\Q)\otimes_\Q \Q(t). \]
\[\begin{split}
&H_3(\bConf_2(\wM)_\Z)\otimes_\Lambda \Q(t)=\Q(t)[ST(K)],\\
&H_2(\bConf_2(\wM)_\Z)\otimes_\Lambda \Q(t)=\Q(t)[ST(*)],
\end{split}
\]
where $ST(K)$ is the restriction of the $S^2$-bundle $ST(M)$ on $K$.
\end{Thm}
Consider the exact sequence
\[ H_4(\bConf_2(\wM)_\Z,\partial\bConf_2(\wM)_\Z)\otimes_\Lambda\Q(t)
\stackrel{\partial}{\to} H_3(\partial\bConf_2(\wM)_\Z)\otimes_\Lambda\Q(t)
\stackrel{i_*}{\to} H_3(\bConf_2(\wM)_\Z)\otimes_\Lambda\Q(t),
\]
where $i_*$ is the map induced by the inclusion\footnote{Since $\Q(t)$ is a torsion-free $\Lambda$-module, one has the isomorphism $H_i(C_*(X)\otimes_\Lambda\Q(t))\cong H_i(X)\otimes_\Lambda\Q(t)$ of $\Lambda$-modules for any $\Z$-space $X$, by the universal coefficient theorem.}. 

\begin{Thm}[Lescop \cite{Les2}]\label{thm:les2}
Let $\tau:TM\to \R^3\times M$ be a trivialization of $TM$ and let $s_\tau:M\to ST(M)$ be a section induced by $\tau$ that sends $M$ to $\{v\}\times M$ for a fixed $v\in S^2$. Suppose that $s_\tau|_K$ agrees with the unit tangent vectors of $K$. Then\footnote{The sign in the formula (\ref{eq:les2}) seems different from that of \cite{Les2}. This is because the homological action $t$ of the knot in \cite{Les2} is our $t^{-1}$. Note that 
\[ \frac{1+t^{-1}}{1-t^{-1}}+\frac{t^{-1}\Delta'(M)(t^{-1})}{\Delta(M)(t^{-1})}=
-\left(\frac{1+t}{1-t}+\frac{t\Delta'(M)}{\Delta(M)}\right)\]}
\begin{equation}\label{eq:les2}
 i_*[s_{\tau}(M)]=-\left(\frac{1+t}{1-t}+\frac{t\Delta'(M)}{\Delta(M)}\right)i_*[ST(K)] \end{equation}
in $H_3(\bConf_2(\wM)_\Z)\otimes_\Lambda\Q(t)$, where $\Delta(M)$ is the Alexander polynomial of $M$ normalized so that $\Delta(M)(1)=1$ and $\Delta(M)(t^{-1})=\Delta(M)(t)$. Hence, there exists a 4-dimensional $\Q(t)$-chain $Q$ such that
\[ \partial Q=s_{\tau}(M)+\left(\frac{1+t}{1-t}+\frac{t\Delta'(M)}{\Delta(M)}\right)ST(K). \]
\end{Thm}
Lescop calls such a $\Q(t)$-chain $Q$ an {\it equivariant propagator}. The equivariant intersection pairing with $Q$ detects all classes in $H_2(\bConf_2(\wM)_\Z)\otimes_\Lambda\Q(t)=\Q(t)[ST(*)]$. More generally, we will call a 4-dimensional relative $\Q(t)$-cycle $Q$ in $(\bConf_2(\wM)_\Z,\partial \bConf_2(\wM)_\Z)$ such that the boundary condition is satisfied in the homology an equivariant propagator.

\subsection{Fiberwise Morse function}\label{ss:fMF}

In the rest of this paper, let $M$ be a closed oriented Riemannian 3-manifold with $b_1(M)=1$ fibered over $S^1$ and let $\kappa:M\to S^1$ be the projection of the fibration. Suppose that fiber of $\kappa$ is path-connected. A {\it fiberwise Morse function} is a $C^\infty$ function $f:M\to \R$ whose restriction $f_s=f|_{\kappa^{-1}(s)}:\kappa^{-1}(s)\to \R$ is Morse for all $s\in S^1$. A {\it generalized Morse function (GMF)} is a $C^\infty$ function on a manifold with only Morse or birth-death singularities (\cite[Appendix]{Ig1}). A {\it fiberwise GMF} is a $C^\infty$ function $f:M\to \R$ whose restriction $f_s:\kappa^{-1}(s)\to \R$ is a GMF for all $s\in S^1$. A {\it critical locus} of a fiberwise GMF is the subset of $M$ consisting of critical points of $f_s$, $s\in S^1$. We will need an {\it oriented} fiberwise GMF, where a fiberwise GMF is oriented if the bundles of negative eigenspaces of the Hessians along the fibers over all the critical loci are oriented and if each birth-death pair near a birth-death locus has incidence number $1$. 

For a critical locus $p$ of the fiberwise gradient $\xi$ of a fiberwise GMF $f$, we denote by $\wcalD_p=\wcalD_p(\xi)$ and $\wcalA_p=\wcalA_p(\xi)$ the descending manifold loci and the ascending manifold loci respectively. If a pair of different critical loci $p,q$ is such that $\mathrm{ind}\,p=\mathrm{ind}\,q=1$ and if $\xi$ is generic, then $\wcalD_p$ and $\wcalA_q$ may intersect transversally at finitely many values of $\kappa$. The intersection of $\wcalD_p$ and $\wcalA_q$ is then a flow line along $\xi$ between $p$ and $q$. Such an intersection is called a {\it $1/1$-intersection} (generally, $i/j$-intersection \cite{HW}).

\begin{Prop}
There exists an oriented fiberwise Morse function $f:M\to \R$ for the fibration $\kappa:M\to S^1$.
\end{Prop}
\begin{proof}
\begin{figure}
\fig{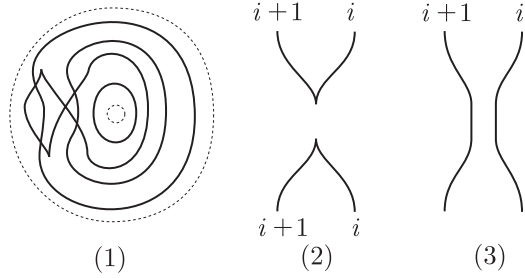}
\caption{Cerf's graphic and birth-death cancellation.}\label{fig:graphic}
\end{figure}
According to Cerf \cite[Ch~I.3]{Ce} or the Framed function theorem of K.~Igusa \cite[Theorem~1.6]{Ig1} (see also \cite[Theorem~4.6.3]{Ig2}), there exists a fiberwise GMF $f:M\to \R$. The graph of critical values of $f_s$ forms a diagram in $\R\times S^1$ (Cerf's {\it graphic}). See Figure~\ref{fig:graphic} (1) for an example. In a graphic, Morse critical loci correspond to arcs and birth-death singularities correspond to beaks. 

If there is a pair of beaks as in Figure~\ref{fig:graphic} (2), we may apply the Birth-death cancellation lemma \cite[Proposition A.2.3]{Ig1} of K.~Igusa to eliminate the pair of beaks by deforming $f$ within the space of smooth functions on $M$ to a fiberwise GMF with less birth-death points, as follows. Let $J=(c,d)\subset S^1$ be a small interval such that a pair $(v_1,v_2)$ of birth-death points as in Figure~\ref{fig:graphic} (2) is included in $\kappa^{-1}(J)$. Suppose that there are no $1/1$-intersections in $\kappa^{-1}(J)$ and that $v_1$ (resp. $v_2$) is a death point (resp. birth point) of index $(0,1)$. The case of birth-death points of index $(1,2)$ is symmetric to this case. The ascending manifold $\calA_{v_i}(\xi)$ is a half-disk with origin $v_i$. We may assume that both $f(v_1)$ and $f(v_2)$ are less than the values of critical loci of index 1. By the sliding technique used in the proof of \cite[Lemma~6.1]{HW}, we may assume that both the ascending manifolds $\calA_{v_1}(\xi)$ and $\calA_{v_2}(\xi)$ do not intersect descending manifolds of other critical points of index 1 and that both $\calA_{v_1}(\xi)$ and $\calA_{v_2}(\xi)$ do not intersect descending manifolds of critical loci of index 2 except for one common critical locus $p$ of index 2. There is a regular level surface locus $T\subset \kappa^{-1}(J)$ of $f$ that lie just above $v_1$ and $v_2$, and there is an arc $c$ in $T$ such that 
\begin{enumerate}
\item $\kappa|_c:c\to J$ is a submersion,
\item $c$ is included in $\wcalD_{p}(\xi)\cap T$,
\item the 1-disks $\calA_{v_1}(\xi)\cap T$ and $\calA_{v_2}(\xi)\cap T$ are included in a small neighborhood $U$ of $c$ in $T$.
\end{enumerate}
Then we may find a smoothly embedded arc $c'$ in a small neighborhood of $T$ in $\kappa^{-1}(J)$ connecting $v_1$ and $v_2$, most of which is included in $U$ and such that $\kappa|_{c'}$ is a submersion. Along this arc, the Birth-death cancellation lemma can be achieved. Here, for orientability, one may need to slide the descending or ascending manifold of $v_1$ or $v_2$ before the surgery so that the orientations become compatible. For example, if $v_1$ is a death point of index $(1,2)$, we may assume, after an appropriate shift, that $f(v_1)$ is greater than the values of $f$ of all the critical loci of index 1 in the $\kappa$-level of $v_1$. Moreover, we may assume that the boundary line $\ell$ of the half-disk $\calD_{v_1}(\xi)$ with origin $v_1$ intersects a level surface $f^{-1}(f(v_1)-\ve)$, for $\ve>0$ small, in two points on the same circle $A\subset f^{-1}(f(v_1)-\ve)$. After a perturbation of $\xi$ inducing a rotation of $\ell$ on $A$ and taking back $v_1$ to the initial position, we may arrange that the orientations of $\calD_{v_1}(\xi)$ and $\calD_{v_2}(\xi)$ are as desired.

Finally, we must check that any beaks in a graphic can be arranged to form pairs of beaks as in Figure~\ref{fig:graphic}. This follows from the Beak lemma of Cerf (\cite[Ch.~IV, \S{3}]{Ce}, see also \cite[Theorem~1.3]{La}). Hence all beaks can be eliminated and the result is as desired.
\end{proof}

\subsection{Amidakuji-like paths}\label{ss:def-al-path}

We fix an oriented fiberwise Morse function $f:M\to \R$ on $M$ and its gradient $\xi$ along the fibers that satisfies the parametrized Morse--Smale condition, i.e., the descending manifold loci and the ascending manifold loci are mutually transversal in $M$. Let $\wf:\wM\to \R$ denote the $\Z$-invariant lift of $f$ and let $\widetilde{\xi}$ denote the lift of $\xi$. We say that a piecewise smooth embedding $\sigma:[\mu,\nu]\to \wM$ is {\it descending} if $\widetilde{\kappa}(\sigma(\mu))\geq \widetilde{\kappa}(\sigma(\nu))$ and $\wf(\sigma(\mu))\geq \wf(\sigma(\nu))$. We say that $\sigma$ is {\it horizontal} if $\mathrm{Im}\,\sigma$ is included in a single fiber of $\widetilde{\kappa}$ and say that $\sigma$ is {\it vertical} if $\mathrm{Im}\,\sigma$ is included in a critical locus of $\wf$.

\begin{Def}\label{def:al-path}
Let $x,y$ be two points of $\widetilde{M}$ such that $\widetilde{\kappa}(x)\geq\widetilde{\kappa}(y)$. An {\it amidakuji-like path, or an AL-path, from $x$ to $y$} is a sequence $\gamma=(\sigma_1,\sigma_2,\ldots,\sigma_n)$, where
\begin{enumerate}
\item for each $i$, $\sigma_i$ is a descending embedding $[\mu_i,\nu_i]\to \widetilde{M}$ for some real numbers $\mu_i,\nu_i$ such that $\mu_i< \nu_i$, 
\item for each $i$, $\sigma_i$ is either horizontal or vertical with respect to $\wf$,
\item if $\sigma_i$ is horizontal, then $\sigma_i$ is a flow line of $\widetilde{\xi}$, possibly broken at critical loci,
\item $\sigma_1(\mu_1)=x$, $\sigma_n(\nu_n)=y$,
\item $\sigma_i(\nu_i)=\sigma_{i+1}(\mu_{i+1})$ for $1\leq i<n$,
\item if $\sigma_i$ is horizontal (resp. vertical) and if $i<n$, then $\sigma_{i+1}$ is vertical (resp. horizontal)\footnote{AL-path appears very similar to ``flow line with cascades'', defined before in \cite{Fr} in relation to Morse-Bott theory. We think that the purpose and the origin for the two notions are different (see also \S\ref{ss:motivation}). Yet it might be interesting to consider a unified generalization of both notions. 
}.
\end{enumerate}
We say that two AL-paths are {\it equivalent} if they differ only by reparametrizations on segments.
\end{Def}
See Figure~\ref{fig:ex_desc_route} for an example of an AL-path. We remark that we do not allow $\sigma_i$ to be a constant map. For an AL-path $\gamma=(\sigma_1,\sigma_2,\ldots,\sigma_n)$, we write
\[ \mathrm{Im}\,\gamma=\bigcup_{i=1}^n \mathrm{Im}\,\sigma_i. \]

\subsection{Main result}

Let $\acalM_2(\wf)$ be the set of all AL-paths in $\wM$. It will turn out that there is a natural structure of non-compact manifold with corners on $\acalM_2(\wf)$. The $\Z$-action $\gamma\mapsto t^n\cdot\gamma$ on a path induces a free $\Z$-action on $\acalM_2(\wf)$. Let $\acalM_2(\wf)_\Z$ be the quotient of $\acalM_2(\wf)$ by the $\Z$-action. For the fiberwise gradient $\xi$ of $f$, let $\hat{\xi}$ be the nonsingular vector field $\xi+\mathrm{grad}\,\kappa$ on $M$. Let $s_{\hat{\xi}}:M\to ST(M)$ be the normalization $-\hat{\xi}/\|\hat{\xi}\|$ of the section $-\hat{\xi}$. Now we state the main theorem of this paper, which gives an explicit equivariant propagator. 

\begin{Thm}[Theorem~\ref{thm:1}, Corollary~\ref{cor:1}]\label{thm:main}
Let $M$ be the mapping torus of an orientation preserving diffeomorphism $\varphi:\Sigma\to \Sigma$ of closed, connected, oriented surface $\Sigma$. 
\begin{enumerate}
\item There is a natural closure $\bacalM_2(\wf)_\Z$ of $\acalM_2(\wf)_\Z$ that has the structure of a countable union of smooth compact manifolds with corners whose codimension 0 strata are disjoint from each other.
\item Suppose that $\kappa$ induces an isomorphism $H_1(M)/\mathrm{Torsion}\cong H_1(S^1)$ (Lemma~\ref{lem:homology-mtorus}). Let $\bar{b}:\bacalM_2(\wf)_\Z\to \wM\times_\Z\wM$ be the evaluation map. Let $B\ell_{\bar{b}^{-1}(\widetilde{\Delta}_M)}(\bacalM_2(\wf)_\Z)$ denote the blow-up of $\bacalM_2(\wf)_\Z$ along $\bar{b}^{-1}(\widetilde{\Delta}_M)$. Then $\bar{b}$ induces a map $B\ell_{\bar{b}^{-1}(\widetilde{\Delta}_M)}(\bacalM_2(\wf)_\Z)\to \bConf_2(\wM)_\Z$ and it represents a 4-dimensional $\Q(t)$-chain $Q(\wf)$ in $\bConf_2(\wM)_\Z$ that satisfies the identity
\[ [\partial Q(\wf)]=[s_{\hat\xi}(M)]+\frac{t\zeta'_{\varphi}}{\zeta_{\varphi}}[ST(K)] \]
in $H_3(\partial\bConf_2(\wM)_\Z)\otimes_\Lambda\Q(t)$. Thus $Q(\wf)$ is an equivariant propagator. (A precise formula for $\partial Q(\wf)$ in the chain level is given in Theorem~\ref{thm:1}.)
\end{enumerate}
\end{Thm}

The formula in (2) shows that the Lefschetz zeta function $\zeta_\varphi$ can be considered as the obstruction to extending $s_{\hat\xi}(M)$ to a relative $\Q(t)$-cycle in $\bConf_2(\wM)_\Z$. By \cite[Proposition~4.5]{Les2} and by the identity $\Delta(M)=t^{-g(\Sigma)}\det(1-t\varphi_{*1})$ for fibered 3-manifold, the formula in (2) recovers the formula for $\partial Q$ in Theorem~\ref{thm:les2} in $H_3(\partial\bConf_2(\wM)_\Z)\otimes_\Lambda\Q(t)$.

Note that it is not straightforward that the substitution of our chain $Q(\wf)$ into Lescop's formula for the equivariant triple intersection gives rise to a topological invariant of $M$ because in \cite{Les2}, the proof of invariance is essentially based on the assumption that the boundary of equivariant propagator concentrates on $s_\tau(M)$ and on $t^iST(K)$ for a given knot in $M$. Since we consider the union of critical loci of $\xi$ instead of a knot, the identity in Theorem~\ref{thm:main} holds only in the homology. So a modification of the proof is required to get an invariant from our chain $Q(\wf)$. See \cite{Wa2} for detail.

\subsection{Motivation for the definition of AL-path}\label{ss:motivation}

We give an informal illustration of how AL-paths arise naturally, with a simple example. Let $M=S^2\times S^1$ and consider $M$ as the mapping torus of a generic diffeomorphism $\varphi:S^2\to S^2$ isotopic to the identity. More precisely, suppose that $\varphi$ is the gradient flow $\Phi_{\mathrm{grad}\,h}^s$ for a Morse function $h:S^2\to \R$ and for a constant $s>0$. Let $M$ be the quotient space of $S^2\times [0,1]$ obtained by identifying $(x,0)$ with $(\varphi(x),1)$ for all $x\in S^2$. The upward unit vector field $\frac{\partial}{\partial t}$ on $S^2\times [0,1]$ induces an upward nonsingular vector field $\hat{\xi}$ on $M$. Then $\hat{\xi}$ is of the form $\xi_0+\alpha\,\mathrm{grad}\,h$ for a constant $\alpha>0$, where $\xi_0$ is the gradient for the projection $S^2\times S^1\to S^1$ with respect to the product metric. Since $h$ has finitely many critical points, $\hat{\xi}$ has finitely many embedded (i.e., without overlapping) closed orbits. 

If we consider analogous to \cite{Fu}, a candidate for the propagator would be the moduli space
\[ \calM_2(\hat{\xi})=\{(x,y)\in M\times M;\exists\, T>0,y=\Phi_{-\hat{\xi}}^T(x)\}. \]
This is a non-compact 4-dimensional manifold immersed in $M\times M$ in a very complicated way. It is hard to deal with $\calM_2(\hat{\xi})$ in the following sense. To define $\Z$-equivariant version of Chern--Simons perturbation theory, we would like to consider, for example, the triple intersection of the moduli spaces $\calM_2(\hat{\xi}_1), \calM_2(\hat{\xi}_2), \calM_2(\hat{\xi}_3)$ in $M\times M$ for a generic triple $(\xi_1,\xi_2,\xi_3)$ of upward vector fields, which corresponds to the $\Theta$-graph. However, since $\calM_2(\hat{\xi})$ is non-compact, it is unclear that the triple intersection number is well-defined, even if counted for a fixed homotopy class of $\Theta$-graphs in $M$. For example, nullhomotopic $\Theta$-graph may wind around arbitrarily many times in $M$.

\begin{figure}
\fig{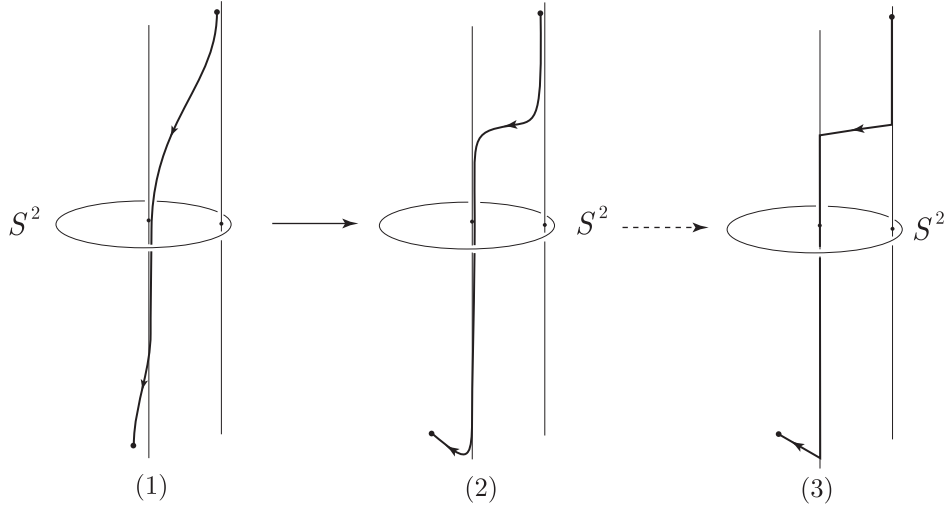}
\caption{An integral curve is approximated by an AL-path. The vertical straight lines are the lifts of closed orbits of $-\hat{\xi}$ of index 1 and 2.}\label{fig:converge_curve}
\end{figure}

Here, Lescop's framework developed in \cite{Les2} is crucial. She considered the equivariant triple intersection of three $\Q(t)$-chains as in Theorem~\ref{thm:les2} and observed that it works nicely for the purpose of constructing an invariant. That the triple intersection is well-defined is then obvious from definition. To utilize Lescop's framework, we considered deforming $\calM_2(\hat{\xi})$ into a $\Q(t)$-chain. 

Now let us return to the example given above. The lift $\hat{\gamma}$ of an integral curve $\gamma$ of $-\hat{\xi}$ in $\wM$ is as in Figure~\ref{fig:converge_curve} (1). Now deform $\hat{\xi}=\xi_0+\alpha\,\mathrm{grad}\,h$ to $\hat{\xi}_\rho=\rho\,\xi_0 + \alpha\,\mathrm{grad}\,h$, where $\rho:M\to \R$ is a smooth function such that $\rho=1$ on each closed orbit of $\hat{\xi}$ and supported on a small neighborhood $N$ of the union of all closed orbits of $-\hat{\xi}$. After the deformation, $\hat{\xi}_\rho$ is horizontal outside $N$ and an integral curve of $-\hat{\xi}_\rho$ will become as in Figure~\ref{fig:converge_curve} (2). If the support of $\rho$ gets very small, an integral curve of $-\hat{\xi}_\rho$ can be approximated by an AL-path, as in Figure~\ref{fig:converge_curve} (3). 

\subsection{Organization}
The rest of this paper is organized as follows. In \S\ref{s:mod_AL}, we define the moduli space $\bacalM_2(\wf)_\Z$ of AL-paths in $\wM$ and study its piecewise smooth structure. In \S\ref{s:coori}, we fix (co)orientation of $\bacalM_2(\wf)_\Z$. In \S\ref{s:chain_mod}, we make the moduli space $\bacalM_2(\wf)_\Z$ into a $\Q(t)$-chain $Q(\wf)$ in the equivariant configuration space. The explicit formula for the boundary of $Q(\wf)$ is given in Theorem~\ref{thm:1}. Then the main theorem follows as a corollary of Theorem~\ref{thm:1}. In Appendix~\ref{s:appendix_1}, a necessary and sufficient condition for a 3-manifold fibered over $S^1$ to have the first Betti number 1, is recalled. In Appendix~\ref{s:mfd_corners}, some definitions on smooth manifolds with corners are recalled. In Appendix~\ref{s:blow-up}, the definition of blow-up is recalled.

\mysection{Moduli space of AL-paths}{s:mod_AL}

In this section, we define the moduli space $\bacalM_2(\wf)_\Z$ of AL-paths in $\wM$ (Definitions~\ref{def:M2(f,rho)}, \ref{def:M2(f)}) and study the piecewise smooth structure of $\bacalM_2(\wf)_\Z$ (Proposition~\ref{prop:M2(f)}).

\subsection{Moduli space $\bacalM_2(\wf)$ of AL-paths in $\widetilde{M}$}\label{ss:M2(f)}

Let $M$, $f:M\to \R$, $\wf:\wM\to \R$ and $\widetilde{\xi}$ be as in \S\ref{ss:def-al-path}. Since the positions of the endpoints of an AL-path are not sufficient to recover an AL-path, even up to equivalence, we consider some sequence of points on an AL-path, which represents internal structure of an AL-path, as in \cite{BH}.

In a graphic of a fiberwise Morse function on $M$, there are intersection points between curves of critical values. We may assume that the intersections of the curves are all in general positions. We call an intersection of two curves in a graphic a {\it level exchange bifurcation}. We assume that $f$ and $\xi$ is generic so that level exchange bifurcations and $1/1$-intersections occur at different parameters in $S^1$. 

Let $u_1,u_2,\ldots,u_r\in S^1={[0,1]}/{\sim}$ be the parameters at which the level exchange bifurcations occur. Choose a small number $\ve>0$ and let $I_{2j-1}=[u_j-\ve,u_j+\ve]$ and $I_{2j}=[u_j+\ve,u_{j+1}-\ve]$, putting $u_{r+1}=u_1+1$. We assume without loss of generality that $u_1-\ve=0$ and that $\ve$ is small so that there are no $1/1$-intersections in $I_{2j-1}$ for all $j$. Then $S^1$ is split into finitely many closed intervals:
\[ S^1={\bigcup_{j=1}^{2r}I_j}\Big/{\sim}. \]
If there are no level exchange bifurcations of $\wf$, we consider $S^1$ as $I_1/{\sim}$, where $I_1=[0,1]$. Moreover, we consider the lifts
\[ I_{j+2rk}=I_j+k\quad (1\leq j\leq 2r,\ k\in\Z)\]
in $\R$. Then $\{I_j;j\in \Z\}$ covers $\R$. By considering the lifts of $M_{I_j}=\kappa^{-1}(I_j)$, the $\Z$-cover $\widetilde{M}$ of $M$ can be sliced into pieces. We shall construct the moduli space $\bacalM_2(\widetilde{f})$ for the $\Z$-invariant lift $\widetilde{f}:\widetilde{M}\to \R$ of $f$ by using the slices of $\widetilde{M}$. 
 
\subsubsection{Cutting $M$ into pieces}

For the covering $\{I_j;j\in\Z\}$ of $\R$ given above, we write $I_j=[a_j,b_j]$ and let $\{p_1,p_2,\ldots,p_N\}$, $p_i=\{p_i(s)\}_{s\in I_j}$, be the set of all critical loci of $\wf|_{M_{I_j}}$ numbered so that $\wf(p_1(a_j))<\wf(p_2(a_j))<\cdots<\wf(p_N(a_j))$. We fix a sufficiently small number $\eta>0$ (see below for how small $\eta$ has to be). We define the submanifolds $\wGamma_i^{(j)}$, $\wW_i^{(j)}$ and $\wL_i^{(j)}$ of $M_{I_j}$ as follows. 

If $j$ is even, then we define smooth functions $\gamma_i:I_j\to \R$ by $\gamma_i(s)=f_s(p_i(s))-\eta$ for $1\leq i\leq N$ and by $\gamma_{N+1}(s)=f_s(p_N(s))+\eta$ for $i=N+1$. We define $\wGamma_i^{(j)}$, $\wW_i^{(j)}$ and $\wL_i^{(j)}$ 
\[ \wGamma_i^{(j)}=\bigcup_{s\in I_j}f_s^{-1}(\gamma_{i+1}(s)),\quad
\wW_i^{(j)}=\bigcup_{s\in I_j}f_s^{-1}[\gamma_i(s),\gamma_{i+1}(s)],\quad
\wL_i^{(j)}=\bigcup_{s\in I_j}f_s^{-1}(\gamma_i(s)). \]

If $j$ is odd, then suppose that $p_{k}$ and $p_{k+1}$ are the critical loci such that $f_s(p_k(s))=f_s(p_{k+1}(s))$ for some $s\in I_j$ and such that $f_{a_j}(p_k(a_j))<f_{a_j}(p_{k+1}(a_j))$ and $f_{b_j}(p_k(b_j))>f_{b_j}(p_{k+1}(b_j))$. We refer to this $k$ as $k_j$. For $i\neq k_j,k_j+1$, $1\leq i\leq N$, we define $\gamma_i(s)=f_s(p_i(s))-\eta$. For $i=k_j$, let $\gamma_{k_j}:I_j\to \R$ be a smooth function such that 
\begin{enumerate}
\item $f_s(p_{k_j}(s))>\gamma_{k_j}(s)$ and $f_s(p_{k_j+1}(s))>\gamma_{k_j}(s)$ for all $s\in I_j$,
\item $\gamma_{k_j}(s)>f_s(p_{k_j-1}(s))$ for all $s\in I_j$.
\item $\gamma_{k_j}(a_j)=f_{a_j}(p_{k_j}(a_j))-\eta$, $\gamma_{k_j}(b_j)=f_{b_j}(p_{k_j+1}(b_j))-\eta$.
\end{enumerate}
For $i=N+1$, we define $\gamma_{N+1}(s)=f_s(p_N(s))+\eta$. For $i\neq k_j+1$, we define $\wL_i^{(j)}=\bigcup_{s\in I_j}f_s^{-1}(\gamma_i(s))$ and
\[ \begin{split}
\wGamma_i^{(j)}&=\left\{\begin{array}{ll}
\bigcup_{s\in I_j}f_s^{-1}(\gamma_{i+1}(s)) & i\neq k_j,k_j+1\\
\bigcup_{s\in I_j}f_s^{-1}(\gamma_{k_j+2}(s)) & i=k_j
\end{array}\right.\\
\wW_i^{(j)}&=\left\{\begin{array}{ll}
\bigcup_{s\in I_j}f_s^{-1}[\gamma_i(s),\gamma_{i+1}(s)] & i\neq k_j,k_j+1\\
\bigcup_{s\in I_j}f_s^{-1}[\gamma_{k_j}(s),\gamma_{k_j+2}(s)] & i=k_j
\end{array}\right.
\end{split} \]

Here, we choose $\eta$ so that the following conditions are satisfied.
\begin{enumerate}
\item $\eta$ is less than $\frac{f_s(p_{k+1}(s))-f_s(p_k(s))}{2}$
 for all $k$ in $1\leq k\leq N-1$ and for all $s\in I_{2j}$.
\item $\eta$ is less than $\frac{f_s(p_{k+1}(s))-f_s(p_k(s))}{2}$
 for all $k$ in $1\leq k\leq N-1$ except $k_{2j-1}$ and for all $s\in I_{2j-1}$
\item $\eta$ is less than $\frac{f_s(p_{k_{2j-1}+2}(s))-f_s(p_{k_{2j-1}}(s))}{2}$ for all $s\in I_{2j-1}$.
\end{enumerate}

\subsubsection{Descending routes in a cellular diagram}

The decomposition of $\wM$ into the pieces $\wW_i^{(j)}$ gives a planar diagram as follows. Let 
\[\begin{split}
C^{(j)}&=\{(x,a_j)\in\R^2;x\in\R\}, j\in \Z,\\
D_i^{(2j)}&=\{(p_i(s)-\eta,s)\in\R^2;s\in I_{2j}\}\mbox{ (for $1\leq i\leq N$)},\\
D_i^{(2j-1)}&=\{(\gamma_i(s),s)\in\R^2;s\in I_{2j-1}\}\mbox{ (for $i\neq k_{2j-1}+1$, $1\leq i\leq N$)},\\
D_{N+1}^{(j)}&=\{(p_N(s)+\eta,s)\in\R^2;s\in I_j\}.
\end{split}\]
The lines $C^{(j)}$ and arcs $D_i^{(j)}$ splits the plane into rectangular regions. We refer to such rectangles as {\it cells}. Let $B_i^{(2j)}$ be the cell surrounded by $C^{(2j)}$, $C^{(2j+1)}$, $D_i^{(2j)}$ and $D_{i+1}^{(2j)}$. Let $B_i^{(2j-1)}$ be the cell surrounded by $C^{(2j-1)}$, $C^{(2j)}$, $D_i^{(2j-1)}$ and $D_{i+1}^{(2j-1)}$ for $i\neq k_{2j-1}, k_{2j-1}+1$. Let $B_{k_{2j-1}}^{(2j-1)}$ be the cell surrounded by $C^{(2j-1)}$, $C^{(2j)}$, $D_{k_{2j-1}}^{(2j-1)}$ and $D_{k_{2j-1}+2}^{(2j-1)}$. We call this kind of cell a {\it double cell}. The set of cells $B_i^{(j)}$ forms a planar diagram. We refer to the set $\{B_i^{(j)}\}$ as the {\it cellular diagram} for $\widetilde{f}$ and denote it by $\mathrm{cd}(\wf)$. For each cell $B_i^{(j)}$ in a cellular diagram, we write
\[ \wGamma(B_i^{(j)})=\wGamma_i^{(j)},\qquad \wW(B_i^{(j)})=\wW_i^{(j)},\qquad \wL(B_i^{(j)})=\wL_i^{(j)}. \] 

\begin{figure}
\fig{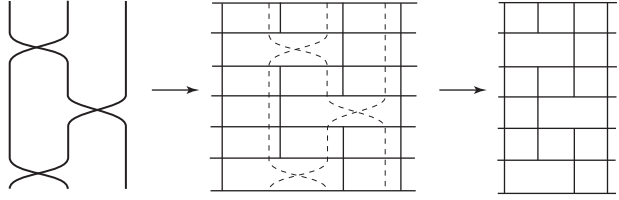}
\caption{From a graphic to a cellular diagram.}\label{fig:cellular_diag.eps}
\end{figure}

We say that a pair $(B,B')$ of cells in the cellular diagram is {\it descending} if it is one of the forms listed in Figure~\ref{fig:BBd}. For $r\geq 1$, we refer to a sequence $(B(1),B(2),\ldots,B(r))$ of cells in the cellular diagram as a {\it descending route} if the following conditions are satisfied.
\begin{enumerate}
\item For each $i$ such that $1\leq i\leq r-1$, the pair $(B(i),B(i+1))$ is descending.
\item For each $i$ such that $1\leq i\leq r-2$, the triple $(B(i),B(i+1),B(i+2))$ is not of the forms shown in Figure~\ref{fig:BBB}.
\end{enumerate}
A descending route serves as a neighborhood of an AL-path.

\begin{figure}
\fig{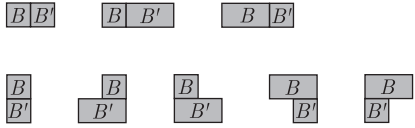}
\caption{Descending pairs $(B,B')$. The longer cells are double cells.}\label{fig:BBd}
\end{figure}

\begin{figure}
\fig{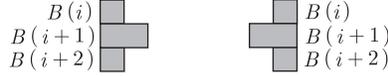}
\caption{Triples $(B(i),B(i+1),B(i+2))$ that are not allowed.}\label{fig:BBB}
\end{figure}

\begin{figure}
\fig{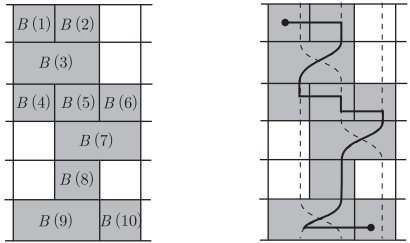}
\caption{An example of a descending route and an AL-path going along it, of rank 3 with 0 breaks. The two horizontal segments in the middle go through $1/1$-intersections.}\label{fig:ex_desc_route}
\end{figure}

\subsubsection{Moduli space of AL-paths going along a descending route}

Let $\rho=(B(1),B(2),\ldots,B(r))$ be a descending route in $\mathrm{cd}(\wf)$. We define the moduli space $\bacalM_2(\widetilde{f};\rho)$ of AL-paths that goes along $\rho$ as follows. Let $X(0)=\wW(B(1))$ and define the space $X(i)$ inductively as
\[ X(i+1)=\left\{\begin{array}{ll}
X(i)\times \wL(B(i)) & \mbox{if $i+1<r$ and if $B(i+1)$ is located }\\
& \mbox{on the right side of $B(i)$}\\
X(i)\times\wL(B(i))\times \wW(B(i+1)) & \mbox{if $i+1=r$ and if $B(r)$ is located}\\
& \mbox{on the right side of $B(i)$}\\
X(i)\times\wW(B(i+1)) & \mbox{if $i+1=r$ and if $B(r)$ is located}\\
& \mbox{on the bottom side of $B(i)$}\\
X(i) & \mbox{otherwise}
\end{array}\right.\]
Then we define $X_\rho=X(r)$ and
\[ Y_\rho=\wW(B(1))\cup\wW(B(2))\cup\cdots\cup\wW(B(r))\subset \widetilde{M}. \]
For example, $X_\rho$ for the descending route in Figure~\ref{fig:ex_desc_route} is
\[ X_\rho=\wW(B(1))\times \wL(B(1))\times\wL(B(4))\times
\wL(B(5))\times\wL(B(9))\times \wW(B(10)). \]
If the image of an AL-path $\gamma=(\sigma_1,\sigma_2,\ldots,\sigma_n)$ is included in $Y_\rho$ and intersects all $\wW(B(i))$ in $Y_\rho$, then we say that $\gamma$ {\it goes along} $\rho$. An example of the plane projection of an AL-path going along a descending route is shown in Figure~\ref{fig:ex_desc_route}. If an AL-path $\gamma$ goes along $\rho$, then the intersection points of $\gamma$ with level surfaces $\wL(B(i))$ together with the endpoints defines a sequence in $X_\rho$. We call such a sequence in $X_\rho$ an {\it AL-sequence going along $\rho$}. An AL-sequence going along $\rho$ recovers an AL-path going along $\rho$ uniquely up to equivalence. We will often represent the equivalence class of an AL-path by an AL-sequence.

\begin{Def}\label{def:M2(f,rho)}
For a descending route $\rho=(B(1),B(2),\ldots,B(r))$, we define the moduli space $\acalM_2(\widetilde{f};\rho)$ as the subspace of $X_\rho$ consisting of AL-sequences going along $\rho$. Let $\bacalM_2(\wf;\rho)$ be the closure of $\acalM_2(\wf;\rho)$ in $X_\rho$. 
\end{Def} 

We remark that for an AL-path $\gamma$, there may be several possibilities for descending routes along which $\gamma$ goes. So we must identify corresponding points in moduli spaces $\bacalM_2(\widetilde{f};\rho)$ for different $\rho$. Two descending routes $\rho$ and $\rho'$ such that there is an AL-path $\gamma$ that goes along both $\rho$ and $\rho'$ are related to each other by a sequence of the following relations between descending routes including $\gamma$.

\begin{enumerate}
\item Vertex relation: For each vertex $v$ of the cellular diagram that is not on the boundary of $\bigcup_{B\in \mathrm{cd}(\wf)}B$, there are two types of descending routes that intersect $v$: one goes through the upper-right cell with respect to $v$ and other goes through the lower-left cell with respect to $v$. There are seven possibilities for such pairs, shown in Figure~\ref{fig:desc_r_vertex}, depending on the types of the three or four cells surrounding a vertex $v$. Vertex relation is a relation between two descending routes that differ only by one of the pairs shown in Figure~\ref{fig:desc_r_vertex}.
\item Edge relation: Edge relation is a relation between two descending routes $\rho$ and $\rho'$ that are one of the following forms:
\[\left\{\begin{array}{l}
\rho=(B_1,B_2,\ldots,B_r), \rho'=(B_1,B_2,\ldots,B_r,B_{r+1}) \mbox{ or vice versa.}\\
\rho=(B_1,B_2,\ldots,B_r), \rho'=(B_2,\ldots,B_r) \mbox{ or vice versa.}
\end{array}\right.\]
\end{enumerate}
We will say that two descending routes are {\it adjacent} if they are related by a vertex relation or an edge relation. We shall define the gluing operation $\cup_{\rho\cap \rho'}$ between the moduli spaces for $\rho$ and $\rho'$ that are adjacent.

\subsubsection{Gluing for a vertex relation}

For a pair $(\rho,\rho')$ of descending routes in the cellular diagram that differ only by one of the pairs in Figure~\ref{fig:desc_r_vertex}, we define the space $X_{\rho\cap\rho'}$ by
\[ X_{\rho\cap\rho'}=\left\{\begin{array}{ll}
X_\rho\cap X_{\rho'} & \mbox{$(\rho,\rho')$ is one of (1)$\sim$(5) in Figure~\ref{fig:desc_r_vertex}}\\
X_\rho\times(\wL(B_c)\cap \wW(B_a)) & \mbox{$(\rho,\rho')$ is (6) in Figure~\ref{fig:desc_r_vertex}}\\
X_{\rho'}\times(\wL(B_c)\cap \wW(B_a)) & \mbox{$(\rho,\rho')$ is (7) in Figure~\ref{fig:desc_r_vertex}}
\end{array}\right. \]
where for (6) $B_a$ is the upper cell and $B_c$ is the lower left cell, for (7) $B_a$ is the lower cell and $B_c$ is the upper left cell. Then there are natural maps
\[ X_{\rho}\stackrel{\iota}{\leftarrow}X_{\rho\cap\rho'}\stackrel{\iota'}{\to}X_{\rho'}. \]
We define $X_{\rho\cup \rho'}=X_\rho\cup_{X_{\rho\cap\rho'}}X_{\rho'}$ as the push-out of this diagram. Let $\eta:X_\rho\to X_{\rho\cup\rho'}$ and $\eta':X_{\rho'}\to X_{\rho\cup \rho'}$ be the natural maps. We define $\cup_{\rho\cap\rho'}$ as the amalgamation
\[ \bacalM_2(\widetilde{f};\rho)\cup_{\rho\cap\rho'}\bacalM_2(\widetilde{f};\rho')
=\eta\bacalM_2(\widetilde{f};\rho)\cup_{\eta\iota\bacalM_2(\widetilde{f};\rho\cap\rho')}\eta'\bacalM_2(\widetilde{f};\rho'),\]
which is a subspace of $X_{\rho\cup\rho'}$.

\begin{figure}
\fig{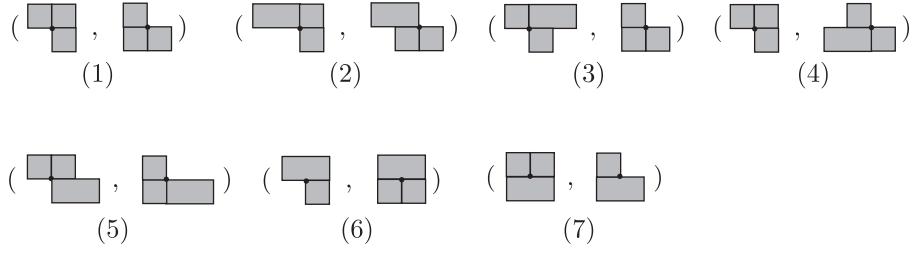}
\caption{Vertex relation}\label{fig:desc_r_vertex}
\end{figure}

\subsubsection{Gluing for an edge relation}

Let $(\rho,\rho')$ be a pair of descending routes in the cellular diagram such that $\rho'$ is obtained from $\rho$ by adding a cell. Suppose that the extra cell is added at the tail of $\rho$. The case that a cell is added at the head of $\rho$ is symmetric to this case. So we put
\[ \rho=(B(1),B(2),\ldots,B(r)),\quad 
\rho'=(B(1),B(2),\ldots,B(r+1)). \]
\begin{enumerate}
\item If $B(r+1)$ is located on the right side of $B(r)$, then $X_\rho=X(r-1)\times \wW(B(r))$ and $X_{\rho'}=X(r-1)\times \wL(B(r))\times \wW(B(r+1))$. Set 
\[ \begin{split}
  \bacalM_2(\widetilde{f};\rho)_{[\rho']}&=\bacalM_2(\widetilde{f};\rho)\cap (X(r-1)\times \wL(B(r))),\\
  \bacalM_2(\widetilde{f};\rho')_{[\rho]}&=\bacalM_2(\widetilde{f};\rho')\cap (X(r-1)\times \wL(B(r))\times \wL(B(r))).
\end{split}\]
The projection $\mathrm{pr}_2:\wL(B(r))\times \wL(B(r))\to \wL(B(r))$ on the second factor induces a homeomorphism
\[ \eta_{\rho,\rho'}:\bacalM_2(\widetilde{f};\rho')_{[\rho]}
\stackrel{\approx}{\to} \bacalM_2(\widetilde{f};\rho)_{[\rho']}. \]

\item If $B(r+1)$ is located on the bottom of $B(r)$, then $X_\rho=X(r-1)\times\wW(B(r))$ and $X_{\rho'}=X(r-1)\times \wW(B(r+1))$. Set
\[ \begin{split}
  \bacalM_2(\widetilde{f};\rho)_{[\rho']}&=\bacalM_2(\widetilde{f};\rho)\cap(X(r-1)\times (\wW(B(r))\cap\wW(B(r+1))),\\
  \bacalM_2(\widetilde{f};\rho')_{[\rho]}&=\bacalM_2(\widetilde{f};\rho')\cap(X(r-1)\times (\wW(B(r))\cap\wW(B(r+1))).
\end{split}\]
The identity map induces a homeomorphism
\[ \eta_{\rho,\rho'}:\bacalM_2(\widetilde{f};\rho')_{[\rho]}
\stackrel{\approx}{\to} \bacalM_2(\widetilde{f};\rho)_{[\rho']}. \]
\end{enumerate}
In both cases, we define $\cup_{\rho\cap\rho'}$ as the amalgamation
\[ \bacalM_2(\widetilde{f};\rho)\cup_{\rho\cap\rho'}\bacalM_2(\widetilde{f};\rho')
=\bacalM_2(\widetilde{f};\rho)\cup_{\eta_{\rho,\rho'}}\bacalM_2(\widetilde{f};\rho'), \]
where the gluing is done between the subspaces $\bacalM_2(\widetilde{f};\rho')_{[\rho]}$ and $\bacalM_2(\widetilde{f};\rho)_{[\rho']}$.

\subsubsection{The definition of $\bacalM_2(\widetilde{f})$, $\bacalM_2(\widetilde{f})_\Z$ and its stratification}\label{ss:def_M2}

\begin{Def}\label{def:M2(f)}
We define
\[ \bacalM_2(\widetilde{f})=\bigcup_{\rho}\bacalM_2(\widetilde{f};\rho), \]
which is the amalgamation generated by $\cup_{\rho\cap \rho'}$ defined above. The group $\Z$ acts on $\bacalM_2(\wf)$ by the diagonal action $(x_1,x_2,\ldots,x_n)\mapsto (tx_1,tx_2,\ldots,tx_n)$ and we define
\[ \bacalM_2(\wf)_\Z=\bacalM_2(\wf)/\Z. \]
\end{Def}

Now we shall describe the natural stratification of $\bacalM_2(\wf)$. We say that an AL-path $\gamma=(\sigma_1,\sigma_2,\ldots,\sigma_n)$ is of {\it rank $r$} if it has $r$ vertical segments and say that it has {\it $m$ breaks} if the sum of times that the horizontal segments are broken is $m$. For a descending route $\rho$, let
\[   S_{r,m}(\widetilde{f};\rho)=\{\mbox{AL-sequences going along $\rho$, rank $r$, $m$ breaks}\}\]
and let $\bS_{r,m}(\widetilde{f};\rho)$ be the closure of $S_{r,m}(\widetilde{f};\rho)$ in $X_\rho$. Let
\[  \bS_{r,m}(\widetilde{f})=\bigcup_{\rho}\bS_{r,m}(\widetilde{f};\rho), \]
where the pieces are glued together by the restrictions of the amalgamations $\cup_{\rho\cap\rho'}$.
Let 
\[ \bT_{r,m}(\wf;\rho)=\bS_{r,m}(\wf;\rho)\cap \partial X_\rho.\]

We define a {\it degeneracy} of an AL-path $\gamma=(\sigma_1,\sigma_2,\ldots,\sigma_n)$ in $\bacalM_2(\wf;\rho)$ as one of the following conditions for $\gamma$.
\begin{enumerate}
\item A horizontal segment in $\gamma$ is broken at a critical locus.
\item $\mathrm{Im}\,\gamma$ intersects a horizontal face in $\partial Y_\rho$, i.e., some segment $\sigma_i$ is included in $\widetilde{\kappa}^{-1}(a_j)$ for some $j$.
\item The condition that $\sigma_1(\mu_1)\in\wGamma(B(1))\cup\wL(B(1))$. 
\item The condition that $\sigma_n(\nu_n)\in\wGamma(B(r))\cup\wL(B(r))$.
\end{enumerate}
We define the number of degeneracy of an AL-path $\gamma\in\bacalM_2(\wf;\rho)$ as the number of counts of the items in the above list that $\gamma$ satisfies. For example, if the first segment $\sigma_1$ in $\gamma$ is horizontal, intersects a horizontal face in $\partial Y_\rho$, has 1 break and if $\gamma$ has no other degeneracy in the list, then the degeneracy of $\gamma$ is 2. We denote by $\deg\,{\gamma}$ the number of degeneracy of $\gamma$. 

Let $b:\bacalM_2(\wf)\to \wM\times\wM$ be the map that assigns to a sequence $(x,\ldots,y)$ the pair of endpoints $(x,y)$. Let $|\rho|$ denote the number of cells in $\rho$. The main proposition of this section is the following.

\begin{Prop}\label{prop:M2(f)}
For a descending route $\rho$ in $\mathrm{cd}(\wf)$, we have
\[ \bacalM_2(\widetilde{f};\rho)=\bigcup_{r=0}^{|\rho|}\bS_{r,0}(\widetilde{f};\rho),\]
where $\bS_{r,0}(\wf;\rho)$ is compact and satisfies the following conditions.
\begin{enumerate}
\item If $|\rho|=1$ or $2$, then $\bS_{r,0}(\wf;\rho)-b^{-1}(\Delta_{\wM})$ is a smooth manifold with corners, whose codimension $q$ stratum consists of AL-sequences $\gamma$ with $\deg\,\gamma=q$. In particular,
\[ \begin{split}
  \partial(\bS_{0,0}(\wf;\rho)-b^{-1}(\Delta_{\wM}))&=\left\{\begin{array}{ll}
  \bS_{0,1}(\wf;\rho)\cup \bT_{0,0}(\wf;\rho)-b^{-1}(\Delta_{\wM}) & \mbox{if $|\rho|=1$}\\
  \bT_{0,0}(\wf;\rho)-b^{-1}(\Delta_{\wM}) & \mbox{if $|\rho|=2$}
  \end{array}\right.\\
  \partial(\bS_{1,0}(\wf;\rho)-b^{-1}(\Delta_{\wM}))&=\left\{\begin{array}{ll}
  \bS_{0,1}(\wf;\rho)\cup \bT_{1,0}(\wf;\rho)-b^{-1}(\Delta_{\wM}) & \mbox{if $|\rho|=1$}\\
  \bT_{1,0}(\wf;\rho)-b^{-1}(\Delta_{\wM}) & \mbox{if $|\rho|=2$}
  \end{array}\right.
\end{split}\]
\item If $|\rho|\geq 3$, then $\bS_{r,0}(\widetilde{f};\rho)$ is a smooth compact manifold with corners, whose codimension $q$ stratum consists of AL-paths $\gamma$ with $\deg\,{\gamma}=q$. In particular,
\[ \partial\bS_{r,0}(\wf;\rho)=\bS_{r,1}(\wf;\rho)\cup\bS_{r-1,1}(\wf;\rho)\cup\bT_{r,0}(\wf;\rho). \]
\end{enumerate}
The codimension 0 strata of $\bS_{r,0}(\wf;\rho)$ for $r=0,1,\ldots,|\rho|$ are disjoint from each other in $\bacalM_2(\wf;\rho)$.
\end{Prop}

We prove Proposition~\ref{prop:M2(f)} in the rest of this section\footnote{The settings and the proof for Proposition~\ref{prop:M2(f)} are rather complicated. Perhaps a good way to believe Proposition~\ref{prop:M2(f)} is to consider $\bacalM_2(\wf)$ as a piecewise smooth approximation of the smooth manifold $\calM_2(\hat{\xi})$ in \S\ref{ss:motivation}.}. In \S\ref{ss:mod_cell}, we will first study the piecewise smooth structure of the moduli space $\bacalM_2(\wf;\rho)$ in the case where $\rho$ has only one cell. We will then reconstruct $\bacalM_2(\wf)_\Z$ in \S\ref{ss:mod_long} from basic pieces by iterated fiber products and study its piecewise smooth structure.

\subsection{Moduli space of AL-paths in a cell}\label{ss:mod_cell}

Here we prove Proposition~\ref{prop:M2(f)} in the case when $|\rho|=1$ (Lemmas~\ref{lem:M2WW}, \ref{lem:M2longcell} (iv)) and prove some lemmas (Lemmas~\ref{lem:M2WL}, \ref{lem:M2LW}, \ref{lem:M2LL}, \ref{lem:M2longcell}) that will be necessary later. 
For subsets $A,B\subset \wW_i^{(j)}$, we write
\[ \begin{split}
  \bacalM_2(\wf;A,B)&=\bacalM_2(\wf;\wW_i^{(j)})\cap(A\times B),\\
  \bS_{r,m}(\wf;A,B)&=\bS_{r,m}(\wf;\wW_i^{(j)})\cap(A\times B),\\
  \bT_{r,m}(\wf;A,B)&=\bT_{r,m}(\wf;\wW_i^{(j)})\cap(A\times B).
\end{split}\]

\subsubsection{Moduli space of AL-paths in a cell on $I_j$, $j$ even}

\begin{Lem}\label{lem:M2WL}
Suppose that $j$ is even. The moduli space $\bacalM_2(\wf;\wW_i^{(j)},\wL_i^{(j)})$ is the union of two smooth manifolds $S=\bS_{0,0}(\wf;\wW_i^{(j)},\wL_i^{(j)})$ and $S'=\bS_{1,0}(\wf;\wW_i^{(j)},\wL_i^{(j)})$ with corners, such that
\begin{enumerate}
\item The codimension $q$ stratum of $S$ consists of AL-paths $\gamma$ with $\deg\,\gamma=q+1$. In particular, $\partial S
=\bS_{0,1}(\wf;\wW_i^{(j)},\wL_i^{(j)})\cup \bT_{0,0}(\wf;\wW_i^{(j)},\wL_i^{(j)})$.
\item The codimension $q$ stratum of $S'$ consists of AL-paths $\gamma$ with $\deg\,\gamma=q+1$. In particular, $\partial S'
=\bS_{0,1}(\wf;\wW_i^{(j)},\wL_i^{(j)})\cup \bT_{1,0}(\wf;\wW_i^{(j)},\wL_i^{(j)})$.
\end{enumerate}
\end{Lem}

\begin{Lem}\label{lem:M2LW}
Suppose that $j$ is even. The moduli space $\bacalM_2(\wf;\wGamma_i^{(j)},\wW_i^{(j)})$ is the union of two smooth manifolds $S=\bS_{0,0}(\wf;\wGamma_i^{(j)},\wW_i^{(j)})$ and $S'=\bS_{1,0}(\wf;\wGamma_i^{(j)},\wW_i^{(j)})$ with corners, such that
\begin{enumerate}
\item The codimension $q$ stratum of $S$ consists of AL-paths $\gamma$ with $\deg\,\gamma=q+1$. In particular, $\partial S
=\bS_{0,1}(\wf;\wGamma_i^{(j)},\wW_i^{(j)})\cup \bT_{0,0}(\wf;\wGamma_i^{(j)},\wW_i^{(j)})$.
\item The codimension $q$ stratum of $S'$ consists of AL-paths $\gamma$ with $\deg\,\gamma=q+1$. In particular, $\partial S'
=\bS_{0,1}(\wf;\wGamma_i^{(j)},\wW_i^{(j)})\cup \bT_{1,0}(\wf;\wGamma_i^{(j)},\wW_i^{(j)})$.
\end{enumerate}
\end{Lem}

\begin{Lem}\label{lem:M2LL}
Suppose that $j$ is even. The moduli space $\bacalM_2(\wf;\wGamma_i^{(j)},\wL_i^{(j)})$ is the union of two smooth manifolds $S=\bS_{0,0}(\wf;\wGamma_i^{(j)},\wL_i^{(j)})$ and $S'=\bS_{1,0}(\wf;\wGamma_i^{(j)},\wL_i^{(j)})$ with corners, such that
\begin{enumerate}
\item The codimension $q$ stratum of $S$ consists of AL-paths $\gamma$ with $\deg\,\gamma=q+2$. In particular, $\partial S
=\bS_{0,1}(\wf;\wGamma_i^{(j)},\wL_i^{(j)})\cup \bT_{0,0}(\wf;\wGamma_i^{(j)},\wL_i^{(j)})$.
\item The codimension $q$ stratum of $S'$ consists of AL-paths $\gamma$ with $\deg\,\gamma=q+2$. In particular, $\partial S'
=\bS_{0,1}(\wf;\wGamma_i^{(j)},\wL_i^{(j)})\cup \bT_{1,0}(\wf;\wGamma_i^{(j)},\wL_i^{(j)})$.
\end{enumerate}
\end{Lem}

\begin{Lem}\label{lem:M2WW}
Suppose that $j$ is even. The moduli space $\bacalM_2(\wf;\wW_i^{(j)})$ is compact and $\bacalM_2(\wf;\wW_i^{(j)})-\Delta_{\wW_i^{(j)}}$ is the union of two noncompact smooth manifolds $S=\bS_{0,0}(\wf;\wW_i^{(j)})-\Delta_{\wW_i^{(j)}}$ and $S'=\bS_{1,0}(\wf;\wW_i^{(j)})-\Delta_{\wW_i^{(j)}}$ with corners, such that
\begin{enumerate}
\item The codimension $q$ stratum of $S$ consists of AL-paths $\gamma$ with $\deg\,\gamma=q$. In particular, $\partial S
=\bS_{0,1}(\wf;\wW_i^{(j)})\cup\bT_{0,0}(\wf;\wW_i^{(j)})-\Delta_{\wW_i^{(j)}}$.
\item The codimension $q$ stratum of $S'$ consists of AL-paths $\gamma$ with $\deg\,\gamma=q$. In particular, $\partial S'
=\bS_{0,1}(\wf;\wW_i^{(j)})\cup\bT_{1,0}(\wf;\wW_i^{(j)})-\Delta_{\wW_i^{(j)}}$.
\end{enumerate}
\end{Lem}

We write $I_j=[a_j,b_j]$. Let $W_i^{(j)}=\widetilde{\kappa}^{-1}(a_j)\cap\wW_i^{(j)}$ and $L_i^{(j)}=\widetilde{\kappa}^{-1}(a_j)\cap\wL_i^{(j)}$. We fix a trivialization $\tau_i^{(j)}:\wW_i^{(j)}\stackrel{\cong}{\to} W_i^{(j)}\times I_j$ and identify both sides through $\tau_i^{(j)}$. For subsets $A,B$ of $\wW_i^{(j)}$, write 
\[ [A\times B]_{\geq}
=\{(x,y)\in A\times B;\widetilde{\kappa}(x)\geq\widetilde{\kappa}(y)\}.\]

\begin{proof}[Proof of Lemma~\ref{lem:M2WL}] We first check the structures of manifolds with corners for $\bS_{0,0}(\wf;\wW_i^{(j)},\wL_i^{(j)})$ and $\bS_{1,0}(\wf;\wW_i^{(j)},\wL_i^{(j)})$. Since AL-sequences in $\bS_{0,0}(\wf;\wW_i^{(j)},\wL_i^{(j)})$ are of rank 0, we have the diffeomorphism
\[ \bS_{0,0}(\wf;\wW_i^{(j)},\wL_i^{(j)})\cong\bcalM_2(f_{a_j};W_i^{(j)},L_i^{(j)})\times I_j \]
induced by $\tau_i^{(j)}$, where $\bcalM_2(f_{a_j};W_i^{(j)},L_i^{(j)})$ is the space of gradient flow-lines in a fiber. Now the assertion for $\bS_{0,0}(\wf;\wW_i^{(j)},\wL_i^{(j)})$ follows from \cite[Lemma~2.14]{Wa1} about the smooth structure of the fiberwise moduli space $\bcalM_2$. 

The assertion for $\bS_{1,0}(\wf;\wW_i^{(j)},\wL_i^{(j)})$ follows immediately from the following identity, which holds by definition.
\[ \bS_{1,0}(\wf;\wW_i^{(j)},\wL_i^{(j)})=\Bigl[(\wcalA_{p_i}\cap \wW_i^{(j)})\times (\wcalD_{p_i}\cap \wL_i^{(j)})\Bigr]_{\geq}. \]

That $\bacalM_2(\wf;\wW_i^{(j)},\wL_i^{(j)})=\bS_{0,0}(\wf;\wW_i^{(j)},\wL_i^{(j)})\cup \bS_{1,0}(\wf;\wW_i^{(j)},\wL_i^{(j)})$ follows immediately from the definition of AL-path.
\end{proof}

\begin{proof}[Proof of Lemma~\ref{lem:M2LW}] By symmetry, the proof is completely analogous to that of Lemma~\ref{lem:M2WL}.
\end{proof}

\begin{proof}[Proof of Lemma~\ref{lem:M2LL}] By definition, we have
\[ \bacalM_2(\wf;\wGamma_i^{(j)},\wL_i^{(j)})=\bacalM_2(\wf;\wW_i^{(j)},\wL_i^{(j)})\cap (\wGamma_i^{(j)}\times \wL_i^{(j)}). \]
The intersection is strata transversal. The assertion follows immediately from Lemma~\ref{lem:M2WL}. 
\end{proof}

\begin{proof}[Proof of Lemma~\ref{lem:M2WW}]
We have 
\[ \bS_{0,0}(\wf;\wW_i^{(j)})\cong \bcalM_2(f_{a_j};W_i^{(j)})\times I_j. \]
The assertion for $\bS_{0,0}(\wf;\wW_i^{(j)})$ follows from \cite[Lemma~2.15]{Wa1}. The assertion for $\bS_{1,0}(\wf;\wW_i^{(j)})$ follows from the following identity.
\[ \bS_{1,0}(\wf;\wW_i^{(j)})=\Bigl[(\wcalA_{p_i}\cap \wW_i^{(j)})\times (\wcalD_{p_i}\cap \wW_i^{(j)})\Bigr]_{\geq}. \]
\end{proof}

\subsubsection{Moduli space of AL-paths in a cell on $I_j$, $j$ odd}

\begin{Lem}\label{lem:M2longcell} Suppose that $j$ is odd and put $k=k_j$. Then the following hold.
\begin{enumerate}
\item[(i)] $\bacalM_2(\wf;\widetilde{W}_k^{(j)},\widetilde{L}_k^{(j)})$ is the union of two smooth manifolds $\bS_{0,0}(\wf;\wW_k^{(j)},\wL_k^{(j)})$ and $\bS_{1,0}(\wf;\wW_k^{(j)},\wL_k^{(j)})$ with corners. 
\item[(ii)] $\bacalM_2(\wf;\wGamma_k^{(j)},\widetilde{W}_k^{(j)})$ is the union of two smooth manifolds $\bS_{0,0}(\wf;\wGamma_k^{(j)},\wW_k^{(j)})$ and $\bS_{1,0}(\wf;\wGamma_k^{(j)},\wW_k^{(j)})$ with corners. 
\item[(iii)] $\bacalM_2(\wf;\wGamma_k^{(j)},\widetilde{L}_k^{(j)})$ is the union of two smooth manifolds $\bS_{0,0}(\wf;\wGamma_k^{(j)},\wL_k^{(j)})$ and $\bS_{1,0}(\wf;\wGamma_k^{(j)},\wL_k^{(j)})$ with corners.
\item[(iv)] $\bacalM_2(\wf;\wW_k^{(j)})-\Delta_{\wW_k^{(j)}}$ is the union of two smooth manifolds $\bS_{0,0}(\wf;\wW_k^{(j)})-\Delta_{\wW_k^{(j)}}$ and $\bS_{1,0}(\wf;\wW_k^{(j)})-\Delta_{\wW_k^{(j)}}$ with corners. 
\end{enumerate}
The boundaries are similar to Lemmas~\ref{lem:M2WL}, \ref{lem:M2LW}, \ref{lem:M2LL} and \ref{lem:M2WW} respectively. 
\end{Lem}
\begin{proof}
Since the descending manifold loci and the ascending manifold loci between the two different critical loci in $\wW_k^{(j)}$ are disjoint, the set of AL-paths that are close to one critical locus is disjoint from that of AL-paths that are close to another critical locus. So the smooth structures on the corners can be studied separately. The rest of the proof is similar to those for Lemmas~\ref{lem:M2WL}, \ref{lem:M2LW}, \ref{lem:M2LL} and \ref{lem:M2WW}. For the detail, see \cite[Lemma~5.8]{Wa1}.
\end{proof}

\subsection{Piecewise smooth structure of the moduli space of long paths}\label{ss:mod_long}

Here, we shall prove Proposition~\ref{prop:M2(f)} for the cases when $|\rho|\geq 2$. For a descending route $\rho=(B(1),\ldots,B(r))$ in $\mathrm{cd}(\wf)$, let $\underline{\rho}=\{B(1),\ldots,B(r)\}$ denote the set of all cells in $\rho$. We define the equivalence relation $\sim$ on $\underline{\rho}$ generated by the following relations: $B(i)\sim B(i+1)$ if $B(i+1)$ is located on the bottom of $B(i)$. We call an equivalence class in ${\underline{\rho}}/{\sim}$ a {\it block}. Roughly, a block consists of vertically clustered sequence of cells that are successive in $\rho$. We write $\mathrm{blo}(\rho)=|{\underline{\rho}}/{\sim}|$, the number of different blocks in ${\underline{\rho}}/{\sim}$. For example, the descending route given in Figure~\ref{fig:ex_desc_route} has five blocks.

We fix some notations. Let $\rho=(B(1),B(2),\ldots,B(r))$ be a descending route. Let 
\[ \begin{split}
  \bacalM_2(\wf;\rho,\wL(B(r)))&=\{(x,\ldots,y)\in\bacalM_2(\wf;\rho);y\in \wL(B(r))\}\\
  \bacalM_2(\wf;\wGamma(B(1)),\rho)&=\{(x,\ldots,y)\in\bacalM_2(\wf;\rho);x\in \wGamma(B(1))\}\\
  \bacalM_2(\wf;\wGamma(B(1)),\rho_1,\wL(B(r)))&=\bacalM_2(\wf;\wGamma(B(1)),\rho)\cap\bacalM_2(\wf;\rho,\wL(B(r))).
\end{split} \]
Let 
\[ \begin{split}
  \bS_{r,m}(\wf;\rho,\wL(B(r)))&=\bS_{r,m}(\wf;\rho)\cap\bacalM_2(\wf;\rho,\wL(B(r))),\\
  \bS_{r,m}(\wf;\wGamma(B(1)),\rho)&=\bS_{r,m}(\wf;\rho)\cap\bacalM_2(\wf;\wGamma(B(1)),\rho_1),\\
  \bS_{r,m}(\wf;\wGamma(B(1)),\rho,\wL(B(r)))&=\bS_{r,m}(\wf;\rho)\cap\bacalM_2(\wf;\wGamma(B(1)),\rho,\wL(B(r))).
\end{split}\]

\subsubsection{Moduli space of AL-paths in a block}

Here, we consider the case $\mathrm{blo}(\rho)=1$. The case $|\rho|=1$ has been considered in \S\ref{ss:mod_cell}. The following lemma proves Proposition~\ref{prop:M2(f)} in the case $\mathrm{blo}(\rho)=1$ and $|\rho|\geq 2$.

\begin{Lem}\label{lem:M2block}
Suppose that $\mathrm{blo}(\rho)=1$ and $|\rho|\geq 2$. Then
\[ \bacalM_2(\wf;\rho)=\left\{\begin{array}{ll}
\bS_{1,0}(\wf;\rho) & \mbox{if $|\rho|\geq 3$}\\
\bS_{1,0}(\wf;\rho)\cup\bT_{0,0}(\wf;\rho) & \mbox{if $|\rho|=2$}
\end{array}\right. \]
where
\begin{enumerate}
\item $\bS_{1,0}(\wf;\rho)$ is a smooth compact manifold with corners, whose codimension $q$ stratum consists of AL-sequences $\gamma$ with $\deg\,\gamma=q$. In particular, $\partial\bS_{1,0}(\wf;\rho)=\bT_{1,0}(\wf;\rho)$.
\item If $|\rho|=2$, then $\bT_{0,0}(\wf;\rho)-\Delta_{\wW(B(1))\cap\wW(B(2))}$ is a smooth manifold with corners, whose codimension $q$ stratum consists of AL-sequences $\gamma$ with $\deg\,\gamma=q+1$. In particular, $\partial(\bT_{0,0}(\wf;\rho)-\Delta_{\wW(B(1))\cap\wW(B(2))})=\bT_{0,1}(\wf;\rho)-\Delta_{\wW(B(1))\cap\wW(B(2))}$.
\end{enumerate}
\end{Lem}
\begin{proof}
Suppose that $\rho$ has $r$ cells: $\rho=(B(1),\ldots,B(r))$. Then we have $X_\rho=\wW(B(1))\times\wW(B(r))$ since $\mathrm{blo}(\rho)=1$.

If $|\rho|\geq 3$, there are no rank 0 AL-paths going along $\rho$. Thus there are only AL-paths going along $\rho$ of rank 1 and we have
\[ \bacalM_2(\wf;\rho)=\bS_{1,0}(\wf;\rho)=(\wcalA_p\cap\wW(B(1)))\times (\wcalD_{p}\cap\wW(B(r))), \]
where $p$ is the critical locus of $\wf$ that intersects all the cells in $\rho$. The assertion of the lemma follows immediately from this identity.

If $|\rho|=2$, there may be AL-paths going along $\rho$ of rank 0, which goes within $\wW(B(1))\cap \wW(B(2))$. These contribute to the term $\bT_{0,0}(\wf;\rho)$. The assertion about the boundary of $\bT_{0,0}(\wf;\rho)-\Delta_{\wW(B(1))\cap\wW(B(2))}$ is analogous to Lemma~\ref{lem:M2WW}.
\end{proof}

\begin{Cor}\label{cor:M2block_L}
Let $*$ be one of $(\rho,\wL(B(r)))$, $(\wGamma(B(1)),\rho)$ or $(\wGamma(B(1)),\rho,\wL(B(r)))$. If $\mathrm{blo}(\rho)=1$ and $|\rho|\geq 2$, then
\[ \bacalM_2(\wf;*)=\left\{\begin{array}{ll}
\bS_{1,0}(\wf;*) & \mbox{if $|\rho|\geq 3$}\\
\bS_{1,0}(\wf;*)\cup\bT_{0,0}(\wf;*) & \mbox{if $|\rho|=2$}
\end{array}\right. \]
and their boundaries are the intersections with $\partial\bS_{1,0}(\wf;\rho)=\bT_{0,1}(\wf;\rho)$.
\end{Cor}

\subsubsection{Moduli space of AL-paths in several blocks}\label{ss:mod_several_blocks}

The following lemma proves the rest of Proposition~\ref{prop:M2(f)}, i.e., for the case $\mathrm{blo}(\rho)\geq 2$.

\begin{Lem}\label{lem:M2(Wk,Wj)}
If $\mathrm{blo}(\rho)\geq 2$, the moduli space $\bacalM_2(\wf;\rho)$ is the union 
\[ \bigcup_{r=0}^{\mathrm{blo}(\rho)} \bS_{r,0}(\wf;\rho), \]
where $\bS_{r,0}(\wf;\rho)$ is a smooth compact manifold with corners, whose codimension $q$ stratum consists of AL-sequences $\gamma$ with $\deg\,\gamma=q$. In particular,
\[ \begin{split}
\partial\bS_{r,0}(\wf;\rho)
&=\bS_{r,1}(\wf;\rho)\cup\bS_{r-1,1}(\wf;\rho)\cup \bT_{r,0}(\wf;\rho).
\end{split}\]
\end{Lem}

To prove Lemma~\ref{lem:M2(Wk,Wj)}, we prove the following lemma by induction on $\mathrm{blo}(\rho)$.

\begin{Lem}\label{lem:M2(Wk,Lj)}
If $\mathrm{blo}(\rho)\geq 2$, the moduli space $\bacalM_2(\wf;\rho,\wL(B(r)))$ is the union 
\[ \bigcup_{r=0}^{\mathrm{blo}(\rho)} \bS_{r,0}(\wf;\rho,\wL(B(r))), \]
where $\bS_{r,0}(\wf;\rho,\wL(B(r)))$ is a smooth compact manifold with corners, whose codimension $q$ stratum consists of AL-sequences $\gamma$ with $\deg\,\gamma=q+1$. In particular,
\[ \begin{split}
\partial\bS_{r,0}(\wf;\rho,\wL(B(r)))
&=\bS_{r,1}(\wf;\rho,\wL(B(r)))\cup\bS_{r-1,1}(\wf;\rho,\wL(B(r)))\\
&\quad\cup \bT_{r,0}(\wf;\rho,\wL(B(r))).
\end{split}\]
\end{Lem}

Let us consider the case $\mathrm{blo}(\rho)=2$. The following lemma holds.

\begin{Lem}\label{lem:M2(WLLL)}
Suppose that a descending route $\rho=(B(1),\ldots,B(r))$ is such that
\[\begin{split}
\underline{\rho}/{\sim}\,&=\{\underline{\rho_1},\underline{\rho_2}\},\\
\rho_1&=(B(1),\ldots,B(k)),\quad \rho_2=(B(k+1),\ldots,B(r)).
\end{split}\]
Then the projections $i_2:\bacalM_2(\wf;\rho_1,\wL(B(k)))\to \wL(B(k))$ and \\$i_1:\bacalM_2(\wf;\wL(B(k)),\rho_2,\wL(B(r)))\to \wL(B(k))$ are strata transversal (Definition~\ref{def:mfd_corners}). Hence the fiber product
\[\begin{split}
&\bacalM_2(\wf;\rho_1,\wL(B(k)))\times_{\wL(B(k))}\bacalM_2(\wf;\wL(B(k)),\rho_2,\wL(B(r)))\\
&\subset \wW(B(1))\times\wL(B(k))\times\wL(B(k))\times\wL(B(r))
\end{split} \]
is the union of smooth manifolds with corners. The codimension $q$ stratum of the fiber product is
\[ \bigcup_{{q_1+q_2=q}\atop{q_1,q_2\geq 0}}\partial_{q_1}\bacalM_2(\wf;\rho_1,\wL(B(k)))
\times_{\wL(B(k))} \partial_{q_2}\bacalM_2(\wf;\wL(B(k)),\rho_2,\wL(B(r))), \]
where we denote by $\partial_iS$ the codimension $i$ stratum of a stratified space $S$. Hence the codimension $q$ stratum consists of AL-paths that are compositions of $\gamma_1\in\bacalM_2(\wf;\rho_1,\wL(B(k)))$ and $\gamma_2\in\bacalM_2(\wf;\wL(B(k)),\rho_2,\wL(B(r)))$ with $\deg\,\gamma_1+\deg\,\gamma_2=q+3$.
\end{Lem}
Lemma~\ref{lem:M2(WLLL)} can be proved by an argument analogous to \cite[Lemma~2.21]{Wa1} and by using the parametrized Morse--Smale condition for $\widetilde{\xi}$, Corollary~\ref{cor:M2block_L} and Proposition~\ref{prop:BT}. 

\begin{Lem}\label{lem:M2(WLL)}
Let $\rho$ be as in Lemma~\ref{lem:M2(WLLL)}. The space $\bacalM_2(\wf;\rho,\wL(B(r)))$ agrees with
\[ \mathrm{pr}\Bigl[\bacalM_2(\wf;\rho_1,\wL(B(k)))\times_{\wL(B(k))}\bacalM_2(\wf;\wL(B(k)),\rho_2,\wL(B(r)))\Bigr], \]
where $\mathrm{pr}:\bacalM_2(\wf;\rho_1,\wL(B(k)))\times_{\wL(B(k))}\bacalM_2(\wf;\wL(B(k)),\rho_2,\wL(B(r)))\to \wW(B(1))\times\wL(B(k))\times\wL(B(r))$ is the projection $(x,z_k,z_k,y)\mapsto (x,z_k,y)$, which is an embedding. Hence the codimension $q$ stratum of $\bacalM_2(\wf;\rho,\wL(B(r)))$ consists of AL-paths that are compositions of $\gamma_1\in\bacalM_2(\wf;\rho_1,\wL(B(k)))$ and $\gamma_2\in\bacalM_2(\wf;\wL(B(k)),\rho_2,\wL(B(r)))$ with degeneracy $q+1$.
\end{Lem}
Lemma~\ref{lem:M2(WLL)} can be proved by an argument analogous to \cite[Lemma~2.22]{Wa1}.

\begin{Lem}
Suppose that Lemma~\ref{lem:M2(Wk,Lj)} holds true for $\mathrm{blo}(\rho)=p$. Then Lemma~\ref{lem:M2(Wk,Lj)} holds true for $\mathrm{blo}(\rho)=p+1$.
\end{Lem}
\begin{proof}
The proof is analogous to \cite[Lemma~2.23]{Wa1}.
\end{proof}

\begin{proof}[Proof of Lemma~\ref{lem:M2(Wk,Wj)}]
$\bacalM_2(\wf;\rho)$ is the image of the projection from the fiber product
\[ \bacalM_2(\wf;\rho_1,\wL(B(r-1)))\times_{\wL(B(r))}\bacalM_2(\wf;\wL(B(r-1)),\wW(B(r))), \]
where $\rho_1=(B(1),B(2),\ldots,B(r-1))$. It follows from Lemmas~\ref{lem:M2(Wk,Lj)}, \ref{lem:M2LW} and an argument analogous to \cite[Lemma~2.22]{Wa1} that $\bacalM_2(\wf;\rho)$ is the union of smooth compact manifolds with corners. The boundary strata of $\bacalM_2(\wf;\rho)$ given by the iterated fiber product having a term of the form $\bT_{1,0}$ or $\bT_{0,0}$ contributes to $\bT_{r,0}(\wf;\rho)$. The other part contributes to $\bS_{r,1}(\wf;\rho)\cup \bS_{r-1,1}(\wf;\rho)$. 
\end{proof}

\begin{proof}[Proof of Proposition~\ref{prop:M2(f)}]
Proposition~\ref{prop:M2(f)} is a consequence of Lemmas~\ref{lem:M2WW}, \ref{lem:M2longcell}, \ref{lem:M2block} and \ref{lem:M2(Wk,Wj)}. That the codimension 0 strata are disjoint from each other is obvious from the definition.
\end{proof}

\mysection{Convention for (co)orientation of $\bacalM_2(\wf;\rho)$}{s:coori}

Now we define the coorientation of $\bacalM_2(\wf;\rho)$. Suppose that $\rho$ has $m$ blocks and write $\underline{\rho}/{\sim}=\{\underline{\rho_1},\underline{\rho_2},\ldots,\underline{\rho_m}\}$. Then $X_\rho$ is of the form
\[ X_\rho=\wW(B(1))\times \wL_1\times \wL_2\times\cdots\times\wL_{m-1}\times \wW(B(r)), \]
where $\wL_i$ is $\wL(B(k))$ for some $k$. Then it follows from the argument in \S\ref{ss:mod_several_blocks} that the moduli space $\bacalM_2(\wf;\rho)\subset X_\rho$ is given by the union
\begin{equation}\label{eq:union_fiber_products}
 \bigcup_{{i_1,\ldots,i_m}\atop{\in\{0,1\}}}\mathrm{pr}\Bigl[
  \bS_{i_1,0}(\wf;\rho_1,\wL_1)\times_{\wL_1}\bS_{i_2,0}(\wf;\wL_1,\rho_2,\wL_2)\times_{\wL_2}\cdots\times_{\wL_{m-1}}\bS_{i_m,0}(\wf;\wL_{m-1},\rho_m)
  \Bigr],
\end{equation}
where the different pieces are glued together along strata of codimension $\geq 1$. We shall fix the coorientations of the terms in (\ref{eq:union_fiber_products}). Although the ambient manifolds $X_\rho$ change for different $\rho$, it does not matter to compare the orientations of the strata.

\subsection{Coorientations of descending and ascending manifolds}\label{ss:ori_DA}

Let $p$ be a critical locus of $\widetilde{\xi}$, $\wcalD_p$ be the descending manifold locus of $p$ and $\wcalA_p$ be the ascending manifold locus of $p$. Let $x$ be a point of $p$ and let $\Sigma_x$ be the fiber of $\widetilde{\kappa}$ including $x$. By parametrized Morse lemma \cite[\S{A1}]{Ig1}, there is a local coordinate $(x_1,x_2,x_3)$ around $x$ such that $p$ agrees locally with the $x_1$-axis, $x_1$ increases with respect to the height function $\widetilde{\kappa}$, the $x_2x_3$-plane agrees locally with the level surface $\Sigma_x$ of $\widetilde{\kappa}$, and $\wf$ is of the form $c(x_1)\pm x_2^2\pm x_3^2$ for a smooth function $c(x_1)$. 
\begin{enumerate}
\item We define the orientation of $\Sigma_x$ at $x$ by
\[ o(\Sigma_x)_x=\iota(-\mathrm{grad}_x\widetilde{\kappa})\,dx_1\wedge dx_2\wedge dx_3=-dx_2\wedge dx_3. \]
\item Let $\calD_x=\wcalD_p\cap \Sigma_x$. For an arbitrarily given orientation $o(\calD_x)_x$ of $\calD_x$, we define the orientations of $\wcalD_p$ and $\wcalA_p$ by the rule
\[ o(\wcalD_p)_x=-dx_1\wedge o(\calD_x)_x,\quad o(\wcalA_p)_x=-dx_1\wedge o(\calA_x)_x, \]
where we define the orientation of the ascending manifold $\calA_x=\wcalA_p\cap\Sigma_x$ by the rule
\[ o(\calD_x)_x\wedge o(\calA_x)_x=o(\Sigma_x)_x. \]
Then we have
\[ o^*(\wcalD_p)_x=o(\calA_x)_x,\quad o^*(\wcalA_p)_x=(-1)^{\mathrm{ind}\,p}\,o(\calD_x)_x \]
and
\[ o^*(\wcalD_p)_x\wedge o^*(\wcalA_p)_x=o(\calD_x)_x\wedge o(\calA_x)_x=o(\Sigma_x)_x. \]
\end{enumerate}
This definition is independent of the choice of local coordinate.

\subsection{Coorientations of the spaces $\bS_{\ve_i,0}$}\label{ss:coori_S}

We define the coorientations
\[ \begin{split}
  o^*(\bS_{\ve_1,0}(\wf;\wW_i^{(j)}))_{(x,y)}
  &\in \twedge^\bullet T_x^*\wW_i^{(j)}\otimes \twedge^\bullet T_y^*\wW_i^{(j)}\\
  o^*(\bS_{\ve_1,0}(\wf;\rho_1,\wL_1))_{(x,z_1)}
  &\in \twedge^\bullet T_x^*\wW(B(1))\otimes \twedge^{\bullet} T_{z_1}^*\wL_1\\
  o^*(\bS_{\ve_i,0}(\wf;\wL_{i-1},\rho_i,\wL_i))_{(z_{i-1},z_i)}
  &\in \twedge^\bullet T_{z_{i-1}}^*\wL_{i-1}\otimes \twedge^\bullet T_{z_i}^*\wL_i\quad (i\neq 1,m-1)\\
  o^*(\bS_{\ve_m,0}(\wf;\wL_{m-1},\rho_m))_{(z_{m-1},y)}
  &\in \twedge^\bullet T_{z_{m-1}}^*\wL_{m-1}\otimes \twedge^\bullet T_{y}^*\wW(B(r)),
\end{split} \]
where $\mathrm{blo}(\rho_i)=1$, 
as follows. Let $o(M)$ denote a volume form on $M$ giving the orientation and we write $o(\wM)=\pi^* o(M)$. Let 
\[ o(\wL_i)_x=\iota(-\widetilde{\xi}_x)\,o(\wM)_x\quad (x\in \wL_i). \]
The space $S_{0,0}(\wf;\wW_i^{(j)})$ is the restriction of the image of the embedding $\varphi:\wM\times (0,\infty)\to \wM\times \wM$ defined by $\varphi(x,s)=(x,\Phi_{-\widetilde{\xi}}^s(x))$. The Jacobian matrix of $\varphi$ at $(x,y)\in S_{0,0}(\wf;\wW_i^{(j)})$, $y=\Phi_{-\widetilde{\xi}}^s(x)$, is as follows.
\[ (J\varphi)_{(x,y)}=\left(\begin{array}{cc}
  I & O\\
  d\Phi_{-\widetilde{\xi}}^s & -\widetilde{\xi}_y
\end{array}\right) \]
If $T_x\wW_i^{(j)}$ is spanned by an orthonormal basis $e_1,e_2,e_3$, then $T_{(x,y)}S_{0,0}(\wf;\wW_i^{(j)})$ is spanned by $e_1+A(e_1),e_2+A(e_2),e_3+A(e_3), -\widetilde{\xi}_y$, where $A=d\Phi_{-\widetilde{\xi}}^s$. With this in mind, we orient $S_{0,0}(\wf;\wW_i^{(j)})$ as
\[ o(S_{0,0}(\wf;\wW_i^{(j)}))_{(x,y)}=(-d\wf)_y\wedge (dx_1+A_* dx_1)\wedge (dx_2+A_* dx_2)\wedge (dx_3+A_* dx_3), \]
where we assume that $o(\wW_i^{(j)})_x=dx_1\wedge dx_2\wedge dx_3$, where $dx_1,dx_2,dx_3$ is an orthonormal basis of $T_x^*\wW_i^{(j)}$, and $A_*=(d\Phi_{-\widetilde{\xi}}^s)_*:T_x^*\wW_i^{(j)}\to T_y^*\wW_i^{(j)}$. Here, we may assume $dx_1=(-d\wf)_x$ without loss of generality. Then 
\[ o(S_{0,0}(\wf;\wW_i^{(j)}))_{(x,y)}\sim (-d\wf)_y\wedge (-d\wf)_x\wedge (dx_2+A_* dx_2)\wedge (dx_3+A_* dx_3). \]
Similarly, we define
\[ \begin{split}
  &o(S_{0,0}(\wf;\rho_1,\wL_1))_{(x,y)}=(-d\wf)_x\wedge (dx_2+A_* dx_2)\wedge (dx_3+A_* dx_3)\\
  &o(S_{0,0}(\wf;\wL_{m-1},\rho_m))_{(x,y)}=(-d\wf)_y\wedge (dx_2+A_*dx_2)\wedge (dx_3+A_*dx_3)\\
  &o(S_{0,0}(\wf;\wL_{i-1},\rho_i,\wL_i))_{(x,y)}=(dx_2+A_*dx_2)\wedge (dx_3+A_*dx_3)\\
\end{split}\]
For generic points, we define
\[\begin{split}
&o^*(\bS_{0,0}(\wf;\wW_i^{(j)}))_{(x,y)}=*\,o\,(S_{0,0}(\wf;\wW_i^{(j)}))_{(x,y)}\\
&o^*(\bS_{0,0}(\wf;\rho_1,\wL_1))_{(x,z_1)}=*\,o\,(S_{0,0}(\wf;\rho_1,\wL_1))_{(x,z_1)}\\
&o^*(\bS_{0,0}(\wf;\wL_{m-1},\rho_m))_{(z_{m-1},y)}=*\,o\,(S_{0,0}(\wf;\wL_{m-1},\rho_m))_{(z_{m-1},y)}\\
&o^*(\bS_{0,0}(\wf;\wL_{i-1},\rho_i,\wL_i))_{(z_{i-1},z_i)}=*\,o\,(S_{0,0}(\wf;\wL_{i-1},\rho_i,\wL_i))_{(z_{i-1},z_i)}\\
&o^*(\bS_{1,0}(\wf;\wW_i^{(j)}))_{(x,y)}=o^*(\wcalA_{\gamma})_{x}\wedge o^*(\wcalD_\gamma)_{y}\quad (x\in \wcalA_\gamma,y\in\wcalD_\gamma)\\
&o^*(\bS_{1,0}(\wf;\rho_1,\wL_1))_{(x,z_1)}=o^*(\wcalA_{\gamma})_{x}\wedge o^*(\wcalD_\gamma)_{z_1}\quad (x\in \wcalA_\gamma,z_1\in\wcalD_\gamma)\\
&o^*(\bS_{1,0}(\wf;\wL_{m-1},\rho_m))_{(z_{m-1},y)}=o^*(\wcalA_\gamma)_{z_{m-1}}\wedge o^*(\wcalD_\gamma)_{y}\quad (z_{m-1}\in \wcalA_\gamma,y\in\wcalD_\gamma)\\
&o^*(\bS_{1,0}(\wf;\wL_{i-1},\rho_i,\wL_i))_{(z_{i-1},z_i)}=o^*(\wcalA_\gamma)_{z_{i-1}}\wedge o^*(\wcalD_\gamma)_{z_i}\quad (z_{i-1}\in \wcalA_\gamma,z_i\in\wcalD_\gamma),
\end{split}\]
where $*$ is the Hodge star operator (see Remark~\ref{rem:oMM} below) and $\gamma$ is a critical locus that intersects the block. Note that $\wcalD_\gamma$ and $\wcalA_\gamma$  are perpendicular to the level surfaces of $\wf$. Hence $o^*(\wcalD_\gamma)_z,o^*(\wcalA_\gamma)_z\in \twedge^\bullet T_z^*\wL_i^{(j)}$ if $z\in \wL_i^{(j)}$. 

There is no reason that these choices are natural. We fixed these coorientations so that most of the boundaries of the strata of $\bacalM_2(\wf)$ cancel with each other. 

\begin{Rem}\label{rem:oMM}
To determine an orientation of a submanifold $A$ of a manifold $X$ from a coorientation, we need to fix an orientation of the ambient manifold $X$. In this paper, we fix the orientations of $\wM\times \wM$ and $\wM\times \wL_i^{(j)}$ as
\[ \begin{split}
o(\wM\times\wM)_{(x,y)}&=o(\wM)_y\wedge o(\wM)_x,\\
o(\wM\times\wL_i^{(j)})_{(x,y)}&=\iota(-\widetilde{\xi}_y)o(\wM)_y\wedge o(\wM)_x,\\
o(\wL_i^{(j)}\times \wM)_{(x,y)}&=(-1)^3 o(\wM)_y\wedge \iota(-\widetilde{\xi}_x)o(\wM)_x.
\end{split}
 \]
\end{Rem}

\subsection{Coorientation of $\bacalM_2(\wf;\rho)$}

\begin{Def}
We define the coorienatation of the term for $(i_1,\ldots,i_m)$ in (\ref{eq:union_fiber_products}) by 
\[ \begin{split}
  \Pi_{\rho}&\Bigl[
o^*(\bS_{i_1,0}(\wf;\rho_1,\wL_1))_{(x,z_1)}\\
&\otimes o^*(\bS_{i_2,0}(\wf;\wL_1,\rho_2,\wL_3))_{(z_2,z_3)}\otimes\cdots\otimes
o^*(\bS_{i_{m-1},0}(\wf;\wL_{m-2},\rho_{m-1},\wL_{m-1}))_{(z_{m-2},z_{m-1})}\\
&\otimes o^*(\bS_{i_m,0}(\wf;\wL_{m-1},\rho_m))_{(z_{m-1},y)}\Bigr]\\
&\in \twedge^\bullet T_x^*\wW(B(1))
  \otimes \twedge^\bullet T^*_{z_1}\wL_1
  \otimes\cdots
  \otimes \twedge^\bullet T^*_{z_{m-1}}\wL_{m-1}
  \otimes \twedge^\bullet T_y^*\wW(B(r))\\
&\quad=\twedge^\bullet T_{(x,z_1,\ldots,z_{m-1},y)}^*X_\rho,  
\end{split}
 \]
where $\Pi_{\rho}$ is the natural map given by the exterior products $\twedge^\bullet T_{z_i}^*\wL_i\otimes\twedge^\bullet T_{z_i}^*\wL_i\to \twedge^\bullet T_{z_i}^*\wL_i$ at each $\wL_i$. 
\end{Def}

\begin{Exa}\label{exa:1}
We consider the coorientation of $\bacalM_2(\wf;\rho)$ for the descending route in Figure~\ref{fig:ex_desc_route}. Let $\gamma_1,\gamma_2$ and $\gamma_3$ be the three critical loci of $\wf$ that intersect $\wW(B(1)),\wW(B(2))$ and $\wW(B(6))$ respectively. Write $\wW(B(1))=\wW_1$, $\wL(B(1))=\wL_1$, $\wL(B(4))=\wL_4$, $\wL(B(5))=\wL_5$, $\wL(B(9))=\wL_9$ and $\wW(B(10))=\wW_{10}$ for short. Then $X_\rho=\wW_1\times\wL_1\times\wL_4\times\wL_5\times\wL_9\times\wW_{10}$. Put $\rho_1=(B(1))$, $\rho_2=(B(2),B(3),B(4))$, $\rho_3=(B(5))$, $\rho_4=(B(6),B(7),B(8),B(9))$ and $\rho_5=(B(10))$ so that $\underline{\rho}/{\sim}=\{\underline{\rho_1},\underline{\rho_2},\underline{\rho_3},\underline{\rho_4},\underline{\rho_5}\}$. Then by convention in \S\ref{ss:coori_S}, 
\[ \begin{split}
  &o^*(\bS_{0,0}(\wf;\rho_1,\wL_1))_{(x,z_1)}=*\,o\,(S_{0,0}(\wf;\rho_1,\wL_1))_{(x,z_1)},\\
  &o^*(\bS_{1,0}(\wf;\wL_1,\rho_2,\wL_4))_{(z_1,z_4)}=o^*(\wcalA_{\gamma_2})_{z_1}\wedge o^*(\wcalD_{\gamma_2})_{z_4},\\
  &o^*(\bS_{1,0}(\wf;\wL_4,\rho_3,\wL_5))_{(z_4,z_5)}=o^*(\wcalA_{\gamma_1})_{z_4}\wedge o^*(\wcalD_{\gamma_1})_{z_5},\\
  &o^*(\bS_{1,0}(\wf;\wL_5,\rho_4,\wL_9))_{(z_5,z_9)}=o^*(\wcalA_{\gamma_3})_{z_5}\wedge o^*(\wcalD_{\gamma_3})_{z_9},\\
  &o^*(\bS_{0,0}(\wf;\wL_9,\rho_5))_{(z_9,y)}=*\,o\,(S_{0,0}(\wf;\wL_9,\rho_5))_{(z_9,y)}.
\end{split}\]
This gives
\[ \begin{split}
  &o^*(\bS_{3,0}(\wf;\rho))_{(x,z_1,\ldots,z_9,y)}\\
&=*\,o\,(S_{0,0}(\wf;\rho_1,\wL_1))_{(x,z_1)}\wedge o^*(\wcalA_{\gamma_2})_{z_1}\wedge o^*(\wcalD_{\gamma_2})_{z_4}\wedge o^*(\wcalA_{\gamma_1})_{z_4}\wedge o^*(\wcalD_{\gamma_1})_{z_5}\wedge o^*(\wcalA_{\gamma_3})_{z_5}\\
&\quad \wedge o^*(\wcalD_{\gamma_3})_{z_9}\wedge *\,o\,(S_{0,0}(\wf;\wL_9,\rho_5))_{(z_9,y)}\\
&=*\,o\,(S_{0,0}(\wf;\rho_1,\wL_1))_{(x,z_1)}\wedge o^*(\wcalA_{\gamma_2})_{z_1}\wedge \ve_4 o(\wL_4)_{z_4}
\wedge \ve_5 o(\wL_5)_{z_5}\\
&\quad \wedge o^*(\wcalD_{\gamma_3})_{z_9}\wedge *\,o\,(S_{0,0}(\wf;\wL_9,\rho_5))_{(z_9,y)}\\
\end{split}\]
for the signs $\ve_4,\ve_5\in\{-1,1\}$ of the $1/1$-intersections. Here, $*\,o\,(S_{0,0}(\wf;\rho_1,\wL_1))_{(x,z_1)}\wedge o^*(\wcalA_{\gamma_2})_{z_1}$ is equivalent to $o(\wL_x)_x\wedge o^*(\wcalA_{\gamma_2})_{z_1}+\mu\,o^*(\wcalA_{\gamma_2})_x\wedge o(\wL_1)_{z_1}$ for some $\mu>0$, where $\wL_x$ is the level surface of $\wf$ including $x$. Similarly, $o^*(\wcalD_{\gamma_3})_{z_9}\wedge \,*\,o\,(S_{0,0}(\wf;\wL_9,\rho_5))_{(z_9,y)}$ is equivalent to $\mu'\,o(\wL_9)_{z_9}\wedge o^*(\wcalD_{\gamma_3})_y+o^*(\wcalD_{\gamma_3})_{z_9}\wedge o(\wL_y)_y$ for some $\mu'>0$. Hence the evaluations $\mathrm{ev}_i:\bigwedge^\bullet\, T^*_{z_i}\wL_i\to \R$, $\mathrm{ev}_i(\omega)=\omega(f_i)$, with the framings $f_i\in \bigwedge^2\,T_{z_i}\wL_i$ which span the orientations give
\[ (1\otimes\mathrm{ev}_1\otimes \mathrm{ev}_4\otimes \mathrm{ev}_5\otimes \mathrm{ev}_9\otimes 1)(o^*(\bS_{3,0}(\wf;\rho))_{(x,z_1,\ldots,z_9,y)})
\sim \ve_4\ve_5\,o^*(\wcalA_{\gamma_2})_x\wedge o^*(\wcalD_{\gamma_3})_y.
\]
This shows that $o(\bS_{3,0}(\wf;\rho))_{(x,z_1,\ldots,z_9,y)}$ is determined by $o^*(\wcalA_{\gamma_2})_x$, $o^*(\wcalD_{\gamma_3})_y$ and the signs $\ve_4,\ve_5$ of the $1/1$-intersections.
\qed
\end{Exa}

\begin{Lem}\label{lem:coori_consistent}
The orientations induced from those of $\bS_{0,0}(\wf;*)$ and $\bS_{1,0}(\wf;*)$ on the codimension 1 stratum $\bS_{0,1}(\wf;*)$ are opposite where $*=\wW_i^{(j)}, (\rho_1,\wL_1), (\wL_{i-1},\rho_i,\wL_i)$ or $(\wL_{m-1},\rho_m)$.
\end{Lem}
\begin{proof}
  It suffices to compare the orientations at $(x,x')\in \bS_{0,1}(\wf;\wW_i^{(j)},\wL_i^{(j)})$ induced from those of $\bS_{0,0}(\wf;\wW_i^{(j)},\wL_i^{(j)})$ and $\bS_{1,0}(\wf;\wW_i^{(j)},\wL_i^{(j)})$ such that both $x$ and $x'$ are in a small neighborhood of a point $x_0$ on a critical locus $\gamma$ of $\wf$ that intersects $\wW_i^{(j)}$. 

By convention in \S\ref{ss:coori_S},
\begin{equation}\label{eq:coori_WL} 
\begin{split}
  o^*(\bS_{0,0}(\wf;\wW_i^{(j)},\wL_i^{(j)}))_{(x,x')}&=*\,o\,(S_{0,0}(\wf;\wW_i^{(j)},\wL_i^{(j)}))_{(x,x')}\\
  o^*(\bS_{1,0}(\wf;\wW_i^{(j)},\wL_i^{(j)}))_{(x,x')}&=o^*(\wcalA_\gamma)_x\wedge o^*(\wcalD_\gamma\cap \wL_i^{(j)})_{x'}
\end{split}
\end{equation}
We check that the orientations determined by these coorientations induce opposite orientations at the intersection strata $\bS_{0,1}(\wf;\wW_i^{(j)},\wL_i^{(j)})$ (see Lemma~\ref{lem:M2WL}). 

First, we consider the case where $\mathrm{ind}\,\gamma=1$. It suffices to prove the claim when both $x$ and $x'$ are close to $x_0$. By the parametrized Morse lemma \cite[\S{A1}]{Ig1}, there is a local coordinate around $x_0$, say on a neighborhood $U_{x_0}$, such that 
\begin{itemize}
\item $x_0$ corresponds to the origin.
\item $\wcalA_\gamma(\xi)$ agrees with the $x_1x_3$-plane.
\item $\wcalD_\gamma(\xi)$ agrees with the $x_1x_2$-plane.
\end{itemize}
Let $\widetilde{A}=\{(x_1,0,x_3);x_1,x_3\in\R\}\cap U_{x_0}$, $\widetilde{D}=\{(x_1,x_2,0);x_1,x_2\in\R\}\cap U_{x_0}$. Moreover, we may assume for simplicity that $\wL_i^{(j)}\cap U_{x_0}$ agrees with $\wL=\{(x_1,1,x_2);x_1,x_3\in \R\}\cap U_{x_0}$, that $\wf$ agrees with $\widetilde{h}(x_1,x_2,x_3)=-x_2^2+x_3^2$ on $U_{x_0}$ and that $\widetilde{\xi}=\mathrm{grad}\,\widetilde{h}$ on $U_{x_0}$. Then we see that
\[ \bS_{0,0}(\widetilde{h};U_{x_0},\wL)=\{(x_1,x_2,x_3)\times (x_1,1,x_2x_3);x_1,x_3\in\R, \,0\leq x_2\leq 1\}\cap (U_{x_0}\times \wL).\]
This is the image of the embedding $\varphi:U_{x_0}\to U_{x_0}\times \wL$, $\varphi(x_1,x_2,x_3)=(x_1,x_2,x_3)\times (x_1,1,x_2x_3)$. The boundary of $\bS_{0,0}(\widetilde{h};U_{x_0},\wL)$ corresponds to the faces at $x_2=0,1$. The face at $x_2=0$ intersects $\bS_{1,0}(\widetilde{h};U_{x_0},\wL)$ along $\bS_{0,1}(\widetilde{h};U_{x_0},\wL)$. Now we describe the induced orientation at the face at $x_2=0$. Let $dx_1,dx_2,dx_3$ be the standard basis of $T_x^*U_{x_0}$ and we take the standard basis
\[ dx_i=\mathrm{pr}_1^* dx_i,\quad dx_i'=\mathrm{pr}_2^* dx_i\quad (i=1,2,3)\]
of $T_{(x,x')}^*(U_{x_0}\times U_{x_0})$. By \S\ref{ss:coori_S}, we see that $o(\bS_{0,0}(\widetilde{h};U_{x_0},\wL))_{(x,x')}$ is equivalent to
\[ (dx_1+dx_1')\wedge (dx_2+x_3dx_3')\wedge (dx_3+x_2dx_3')=(dx_1+dx_1')\wedge (dx_2+x_3dx_3')\wedge dx_3 \]
at $x_2=0$. Since $dx_2+x_3dx_3'$ is the dual of an inward normal vector to $\bS_{0,0}(\widetilde{h};U_{x_0},\wL)$ at $(x,x')$, we have
\begin{equation}\label{eq:o(dS00)}
 o(\partial \bS_{0,0}(\widetilde{h};U_{x_0},\wL))_{(x,x')}\sim -(dx_1+dx_1')\wedge dx_3\quad (x_2=0).
\end{equation}
On the other hand, by convention of \S\ref{ss:ori_DA} and by Remark~\ref{rem:oMM}, we have
\[ \begin{split}
  o^*(\bS_{1,0}(\widetilde{h};U_{x_0},\wL))_{(x,x')}
&=o^*(\widetilde{A})_x\wedge o^*(\widetilde{D}\cap \wL)_{x'}\\
&=-\alpha\,dx_2\wedge (-\alpha\,dx_3')=dx_2\wedge dx_3' \quad (\mbox{for some $\alpha\in\{-1,1\}$})
\end{split}\]
and this gives
\[ o(\bS_{1,0}(\widetilde{h};U_{x_0},\wL))_{(x,x')}
=-dx_1\wedge dx_3\wedge dx_1'. \]
An inward normal vector to $\bS_{1,0}(\widetilde{h};U_{x_0},\wL)$ at $(x_1,0,x_3)\times (x_1,1,0)$ is a multiple of $(1,0,0,-1,0,0)$. Hence the induced orientation on the boundary is
\[ o(\partial\bS_{1,0}(\widetilde{h};U_{x_0},\wL))_{(x,x')}
=\iota(1,0,0,-1,0,0)(-dx_1\wedge dx_3\wedge dx_1')=(dx_1+dx_1')\wedge dx_3. \]
This is opposite to (\ref{eq:o(dS00)}). 

Next, we consider the case where $\mathrm{ind}\,\gamma=2$. There is a local coordinate around $x_0$, say on a neighborhood $U_{x_0}$, such that 
\begin{itemize}
\item $x_0$ corresponds to the origin.
\item $\wcalA_\gamma(\xi)$ agrees with the $x_1$-axis.
\item $\wcalD_\gamma(\xi)$ agrees with the $x_1x_2x_3$-plane.
\end{itemize}
Let $\widetilde{A}=\{(x_1,0,0);x_1\in\R\}\cap U_{x_0}$, $\widetilde{D}=\{(x_1,x_2,x_3);x_1,x_2,x_3\in\R\}\cap U_{x_0}$. Instead of $\wL_i^{(j)}$, we consider $\wL=\{(x_1,1,x_3);x_1,x_3\in\R\}$ that is tangent to $\wL_i^{(j)}$ at $(x_1,1,0)$. We may assume for simplicity that $\wf$ agrees with $\widetilde{h}'(x_1,x_2,x_3)=-x_2^2-x_3^2$ on $U_{x_0}$ and that $\widetilde{\xi}=\mathrm{grad}\,\widetilde{h}'$ on $U_{x_0}$ We consider the orientation of $\bacalM_2(\widetilde{h}';U_{x_0},\wL)$ at $(x,x')$. Then we see that
\[ \bS_{0,0}(\widetilde{h}';U_{x_0},\wL)=\{(x_1,x_2,x_2x_3)\times (x_1,1,x_3);x_1,x_3\in\R,\, 0\leq x_2\leq 1\}\cap (U_{x_0}\times \wL).\]
This is the image of the embedding $\psi:U_{x_0}\to U_{x_0}\times \wL$, $\psi(x_1,x_2,x_3)=(x_1,x_2,x_2x_3)\times (x_1,1,x_3)$. The boundary of $\bS_{0,0}(\widetilde{h}';U_{x_0},\wL)$ corresponds to the faces at $x_2=0,1$. The face at $x_2=0$ intersects $\bS_{1,0}(\widetilde{h}';U_{x_0},\wL)$ along $\bS_{0,1}(\widetilde{h}';U_{x_0},\wL)$. Now we describe the induced orientation at the face at $x_2=0$. By \S\ref{ss:coori_S}, we see that $o(\bS_{0,0}(\widetilde{h}';U_{x_0},\wL))_{(x,x')}$ is equivalent to
\[ (dx_1+dx_1')\wedge (dx_2+x_3dx_3)\wedge (x_2dx_3+dx_3')=(dx_1+dx_1')\wedge (dx_2+x_3dx_3)\wedge dx_3' \]
at $x_2=0$. Since $dx_2+x_3dx_3$ is the dual of an inward normal vector to $\bS_{0,0}(\widetilde{h}';U_{x_0},\wL)$ at $(x,x')$, we have
\begin{equation}\label{eq:o(dS00)_ind2}
 o(\partial \bS_{0,0}(\widetilde{h}';U_{x_0},\wL))_{(x,x')}\sim -(dx_1+dx_1')\wedge dx_3'\quad (x_2=0).
\end{equation}
On the other hand, by convention of \S\ref{ss:ori_DA} and by Remark~\ref{rem:oMM}, we have
\[ \begin{split}
  o^*(\bS_{1,0}(\widetilde{h}';U_{x_0},\wL))_{(x,x')}
&=o^*(\widetilde{A})_x\wedge o^*(\widetilde{D}\cap \wL)_{x'}\\
&=\alpha\,dx_2\wedge dx_3\wedge (-\alpha)=-dx_2\wedge dx_3\quad (\mbox{for some $\alpha\in\{-1,1\}$})
\end{split}\]
and this gives
\[ o(\bS_{1,0}(\widetilde{h}';U_{x_0},\wL))_{(x,x')}
=dx_1'\wedge dx_3'\wedge dx_1. \]
An inward normal vector to $\bS_{1,0}(\widetilde{h}';U_{x_0},\wL)$ at $(x_1,0,0)\times (x_1,1,x_3)$ is a multiple of $(1,0,0,-1,0,0)$. Hence the induced orientation on the boundary is
\[ o(\partial\bS_{1,0}(\widetilde{h}';U_{x_0},\wL))_{(x,x')}
=\iota(1,0,0,-1,0,0)(dx_1'\wedge dx_3'\wedge dx_1)=(dx_1+dx_1')\wedge dx_3'. \]
This is opposite to (\ref{eq:o(dS00)}). 

The case where $\mathrm{ind}\,\gamma=0$ is the same as the case where $\mathrm{ind}\,\gamma=2$. This completes the proof.
\end{proof}

\mysection{A chain from the moduli space of AL-paths}{s:chain_mod}

In this section, we shall show that the natural map from $\bacalM_2(\wf)_\Z$ to $\wM\times_\Z\wM$ gives a 4-dimensional $\Q(t)$-chain $P(\wf)$ (Lemma~\ref{lem:b_chain}). We consider the blow-up of $P(\wf)$ along the lifts of the diagonal in $\wM\times_\Z\wM$. The result is a 4-dimensional $\Q(t)$-chain $Q(\wf)$ in the equivariant configuration space $\bConf_2(\wM)_\Z$ and give an explicit formula for $\partial Q(\wf)$ (Theorem~\ref{thm:1}). We show that $Q(\wf)$ is in a sense an explicit representative for Lescop's equivariant propagator (Corollary~\ref{cor:1}). 

\subsection{Signs of AL-paths}\label{ss:sign_AL}

Here we define the signs of AL-paths. Let $\Sigma=\widetilde{\kappa}^{-1}(c)$ for $c\in \R$. Let $p,q$ be critical loci of $\widetilde{\xi}$ of the same index that intersect $\Sigma$ at $p_0,q_0$ respectively. The space
\[ \bacalM_2(\wf;p_0,t^iq_0)=\bacalM_2(\wf)\cap{b}^{-1}(p_0\times t^iq_0) \]
is a compact oriented 0-manifold. Thus the natural map 
\[ {b}:\bacalM_2(\wf;p_0,t^iq_0)\to \{p_0\times t^iq_0\} \]
is a finite covering map and represents a 0-dimensional chain in $\wM\times \wM$, which can be written as $n\cdot p_0\otimes t^i q_0$ for an integer $n$. The integer $n$ is determined as follows. As we have seen in Example~\ref{exa:1}, the coorientation of $\bacalM_2(\wf;p_0,t^iq_0)$ considered in $p_0\times\wL_1\times\cdots\times\wL_{m-1}\times t^iq_0$ is given by
\[ \ve_1\,o(\wL_1)_{z_1}\wedge \ve_2\,o(\wL_2)_{z_2}\wedge\cdots\wedge \ve_r\,o(\wL_r)_{z_r}=\ve_1\ve_2\cdots\ve_r\,o(\wL_1)_{z_1}\wedge o(\wL_2)_{z_2}\wedge\cdots\wedge o(\wL_r)_{z_r} \]
for the signs $\ve_i\in \{-1,1\}$. We define the sign of the 0-simplex ${b}(p_0,z_1,z_2,\ldots,z_r,t^iq_0)$ as $\ve_1\ve_2\cdots\ve_r$. Then the integer $n$ is determined as the sum of the signs of all the 0-simplices.

\subsection{Making $\bacalM_2(\wf)_\Z$ into a $\Q(t)$-chain}

\begin{Lem}\label{lem:b_chain}
The natural map $\bar{b}:\bacalM_2(\widetilde{f})_\Z\to \widetilde{M}\times_\Z\widetilde{M}$ gives a 4-dimensional $\Q(t)$-chain. (We denote the $\Q(t)$-chain by $P(\wf)$.)
\end{Lem}
\begin{proof} If a vertical segment $\sigma$ in an AL-path is on a critical locus of $\wf$ of index $i$, then we say that $\sigma$ has index $i$. Let $G_{mn}$ be the subspace of $\acalM_2(\wf)$ consisting of AL-sequences with no breaks such that the numbers of vertical segments of indices 0, 1, 2 are $m$, 0, $n$ respectively. Let $H_{mn}$ be the subspace of $\acalM_2(\wf)$ consisting of AL-sequences with no breaks such that the numbers of vertical segments of indices 0, 2 are $m$, $n$ respectively and that has at least one vertical segments of index 1. Then
\[ \acalM_2(\wf)=\bigcup_{(m,n)}(G_{mn}\cup H_{mn}). \]
Since an AL-path can visit critical loci of index 2 only once and also that of index 0 only once, $(m,n)$ is one of $(0,0),(1,0),(0,1),(1,1)$. Let $\bG_{mn}$ and $\bH_{mn}$ be the closures of $G_{mn}$ and $H_{mn}$ respectively, in $\bacalM_2(\wf)$. Since $\bH_{mn}$ and $\bG_{mn}$ are invariant under the diagonal $\Z$-action, we have the quotients $(\bH_{mn})_\Z$ and $(\bG_{mn})_\Z$ by the $\Z$-action. By Proposition~\ref{prop:M2(f)}, the different pieces are glued along strata of codimension $\geq 1$. It suffices to check that $(\bH_{mn})_\Z$ and $(\bG_{mn})_\Z$ are $\Q(t)$-chains. This is checked in Lemmas~\ref{lem:H_chain} and \ref{lem:G_chain} below.
\end{proof}

\begin{Lem}\label{lem:H_chain}
$\bH_{mn}$ is the union of smooth compact manifolds with corners whose codimension 0 strata are disjoint from each other. The natural map $(\bH_{mn})_\Z\to \wM\times_\Z\wM$ gives a 4-dimensional $\Q(t)$-chain in $\wM\times_\Z\wM$.
\end{Lem}

\begin{proof}
We consider the decomposition $\wM=\bigcup_{i\in \Z}M[i]$ where $M[i]=\widetilde{\kappa}^{-1}[i,i+1]$. Let $\Sigma[i]=\widetilde{\kappa}^{-1}(i)$ for $i\in \Z$. First we consider the simplest case $(m,n)=(0,0)$. Suppose that $k>j$. Let $\bH_{00}(M[k],M[j])$ be the subspace of $\bH_{00}$ consisting of AL-sequences from a point of $M[k]$ to a point of $M[j]$. The critical loci of $\wf$ intersects $\Sigma[k]$ transversally at finitely many points. Let $x_1,x_2,\ldots,x_r\in\Sigma[k]$ be all the intersection points with critical loci of index 1 and let $y_i=t^{k-j-1}x_i\in \Sigma[j+1]$. Let $\gamma_1,\ldots,\gamma_r\subset \wM$ be the critical loci of $\wf$ of index 1 that intersect $x_1,x_2,\ldots,x_r$ respectively and let $c\ell_1=\bigcup_{j=1}^r\gamma_j$. Then $\bH_{00}(M[k],M[j])$ can be written as
\[ \bH_{00}(M[k],M[j])=\coprod_{1\leq p,q\leq r}\bH_{00}(M[k],M[j])_{pq}, \]
where $\bH_{00}(M[k],M[j])_{pq}$ be the subspace of $\bH_{00}(M[k],M[j])$ consisting of AL-sequences $\gamma$ such that
\begin{enumerate}
\item $\mathrm{Im}\,\gamma\cap c\ell_1\cap\Sigma[k]=\{x_p\}$,
\item $\mathrm{Im}\,\gamma\cap c\ell_1\cap\Sigma[j+1]=\{y_q\}$.
\end{enumerate}
Note that there are no $1/1$-intersection on $\Sigma[k]$ and on $\Sigma[j+1]$. Let $\bH_{0*}(M[k])_p$ be the subspace of $\bacalM_2(\wf)$ consisting of AL-sequences from a point in $M[k]$ to the point $x_p$. Let $\bH_{*0}(M[j])_q$ be the subspace of $\bacalM_2(\wf)$ consisting of AL-sequences from the point $y_q$ to a point in $M[j]$. 

Now we decompose the evaluation map $(\bH_{00})_\Z\to \wM\times_\Z\wM$ as a sum of smooth maps from the compact pieces $\bH_{00}(M[k],M[j])_{pq}$. There is a natural homeomorphism
\[ \omega_{pq}(k,j):\bH_{00}(M[k],M[j])_{pq}
\stackrel{\approx}{\to} \bH_{0*}(M[k])_p\times \Omega_{pq}(k,j)\times \bH_{*0}(M[j])_q, \]
where $\Omega_{pq}(k,j)$ is the moduli space of AL-paths from $x_p$ to $y_q$. The evaluation map $\mathrm{ev}_{pq}(k,j):\Omega_{pq}(k,j)\to \{x_p\times y_q\}$ is a finite covering map. We define $\mathrm{pr}_1:\bH_{0*}(M[k])_p\to M[k]$ and $\mathrm{pr}_2:\bH_{*0}(M[j])_q\to M[j]$ as
\[ \mathrm{pr}_1(\gamma_1)=\sigma_1(\mu_1),\quad
\mathrm{pr}_2(\gamma_2)=\sigma_{N'}'(\nu_{N'}'), \]
where $\gamma_1=(\sigma_1,\ldots,\sigma_N)\in \bH_{0*}(M[k])_p$, $\gamma_2=(\sigma_1',\ldots,\sigma_{N'}')\in\bH_{*0}(M[j])_q$, $\sigma_i:[\mu_i,\nu_i]\to M$ and $\sigma_j':[\mu_j',\nu_j']\to M$. The natural map
\[ \varphi_{pq}(k,j):\bH_{00}(M[k],M[j])_{pq}\to M[k]\times M[j] \]
is factorized as $\varphi_{pq}(k,j)=\iota\circ(\mathrm{pr}_1\times\mathrm{ev}_{pq}(k,j)\times\mathrm{pr}_2)\circ \omega_{pq}(k,j)$, where $\iota:M[k]\times\{x_p\times y_q\}\times M[j]\to M[k]\times M[j]$ is the projection map. Then $\varphi_{pq}(k,j)$ can be considered as a chain in $M[k]\times M[j]$. We define the 4-dimensional chain $\Phi(k,j)$ in $M[k]\times M[j]$ by
\[ \Phi(k,j)=\sum_{1\leq p,q\leq r}\varphi_{pq}(k,j). \]
This can be represented by the evaluation map $\bH_{00}(M[k],M[j])\to M[k]\times M[j]$, which gives the endpoints of paths. We define the chain $\Phi$ in $\wM\times_\Z\wM$ by the formal sum
\[ \Phi=\sum_{n=0}^\infty\Phi(n,0), \]
which can be represented by the evaluation map $(\bH_{00})_\Z\to \wM\times_\Z\wM$. 

We check that $\Phi$ is well-defined as a $\Q(t)$-chain, by an analogous argument as \cite{Pa1}. Let $U_p[k]$ be the 2-dimensional chain in $M[k]$ represented by $\mathrm{pr}_1:\bH_{0*}(M[k])_p\to M[k]$ and let $V_q[j]$ be the 2-dimensional chain in $M[j]$ represented by $\mathrm{pr}_2:\bH_{*0}(M[j])_q\to M[j]$. Let $n_{pq}$ be the integer determined by the equation
\[ \mathrm{ev}_{pq}(2,0)_\sharp(\Omega_{pq}(2,0))=n_{pq}(x_p\otimes y_q), \]
along the convention in \S\ref{ss:sign_AL}. Let $A_1$ denote the matrix $(n_{pq})$. Then we have
\[ \begin{split}
  \varphi_{pq}(n,0)&=(A_1^{n-1})_{pq}U_p[n]\otimes V_q[0]=(A_1^{n-1})_{pq}U_p[1]\otimes V_q[-n+1]\\
  &=(A_1^{n-1})_{pq}t^{n-1}(U_p[1]\otimes V_q[0])=(tA_1)_{pq}^{n-1}U_p[1]\otimes V_q[0]. 
\end{split}\]
Therefore,
\[ \begin{split}
\Phi(n,0)&=\sum_{1\leq p,q\leq r}(tA_1)_{pq}^{n-1}U_p[1]\otimes V_q[0] \quad\mbox{ (for $n\geq 1$)},\\
 \sum_{n=0}^\infty\Phi(n,0)
&=\Phi(0,0)+\sum_{1\leq p,q\leq r}(1-tA_1)^{-1}_{pq}U_p[1]\otimes V_q[0]. 
\end{split}\]
Since $(1-tA_1)^{-1}$ is a matrix with entries in $\Q(t)$ and $\Phi(0,0)$ is compact, this shows that $\Phi$ represents a $\Q(t)$-chain in $\wM\times_\Z\wM$.

For $(m,n)=(1,0)$, let $\bH_{00}(\gamma_i,\wM)$ be the subspace of $\bH_{00}$ consisting of AL-sequences from a point of a critical locus $\gamma_i$ of index 2 to $\wM$. Let $\bH_{1*}(\gamma_i,\gamma_i)$ be the subspace of $\bH_{10}$ consisting of AL-sequences between two points in $\gamma_i$. Note that a sequence in $\bH_{1*}(\gamma_i,\gamma_i)$ consists of only the endpoints since the index of $\gamma_i$ is 2. From the result for $\bH_{00}$ above, we see that the restrictions of $\bar{b}$ to $\bH_{00}(\gamma_i,\wM)_\Z$ and $\bH_{1*}(\gamma_i,\gamma_i)_\Z$ give $\Q(t)$-chains in $\gamma_i\times_\Z\wM$ and $\gamma_i\times_\Z\gamma_i$ respectively. Then $(\bH_{10})_\Z$ is the image of the projection from the fiber product
\[ \bH_{1*}(\gamma_i,\gamma_i)\times_{\gamma_i}\bH_{00}(\gamma_i,\wM). \]
The restriction of $\bar{b}$ to the image gives a $\Q(t)$-chain. The case $(m,n)=(0,1)$ is symmetric to this case. 

For $(m,n)=(1,1)$, let $\bH_{00}(\gamma_i,\gamma_j)$ be the subspace of $\bH_{00}$ consisting of AL-sequences from a point in $\gamma_i$ to a point in $\gamma_j$. Then $(\bH_{11})_\Z$ is the image of the projection from the fiber product
\[ \bH_{1*}(\gamma_i,\gamma_i)\times_{\gamma_i}\bH_{00}(\gamma_i,\gamma_j)\times_{\gamma_j}\bH_{*1}(\gamma_j,\gamma_j). \]
The restriction of $\bar{b}$ to the image gives a $\Q(t)$-chain. 
\end{proof}

\begin{Lem}\label{lem:G_chain}
$\bG_{mn}$ is the union of smooth compact manifolds with corners whose codimension 0 strata are disjoint from each other. The natural map $\bar{b}:(\bG_{mn})_\Z\to \wM\times_\Z\wM$ gives a 4-dimensional $\Q(t)$-chain in $\wM\times_\Z\wM$.
\end{Lem}
\begin{proof}
The proof is parallel to Lemma~\ref{lem:H_chain}. The only thing to be checked is that $\bG_{00}$, the one that intersects $\Delta_{\wM}$, is the union of smooth manifolds with corners whose codimension 0 strata are disjoint from each other. By Proposition~\ref{prop:M2(f)}, it suffices to study only the piecewise smooth structure near $\Delta_{\wM}$, in particular, near the diagonal set of a critical locus. We consider only a small neighborhood of a critical locus of index 1 since the cases of other indices are easier than this. By the parametrized Morse lemma \cite[\S{A1}]{Ig1}, it suffices to consider the trivial 1-parameter family of standard Morse functions $h:\R^2\to \R$, $h(x_1,x_2)=-x_1^2+x_2^2$. For simplicity, we assume that $\widetilde{\xi}$ is its gradient with respect to the Euclidean metric, without loss of generality. It follows from the proof of \cite[Lemma~2.15]{Wa1} that the closure of the moduli space of flow lines $\calM_2(h)$ of $h$ is the image of the map
\[ \begin{split}
  &\rho:[0,1]\times \R^1\times\R^1\to \R^2\times \R^2,\\
  &\rho(\eta,a,b)=(\eta a,b)\times (a,\eta b).
\end{split}\]
Let $A$ be the image of $\rho$. Let 
\[ A'=\{(a,b,\eta');a,b\in\R^1,\eta'\in[0,\sqrt{a^2+b^2}]\}. \]
Then the mapping $(\eta a,b)\times (a,\eta b)\mapsto (a,b,\eta\sqrt{a^2+b^2})$ defines a homeomorphism $\alpha:A\to A'$, which is smooth except for the origin. Then one may see that $A'$ is the union of four smooth manifolds with corners, as in Figure~\ref{fig:corner_diag}. This completes the proof.
\begin{figure}
\fig{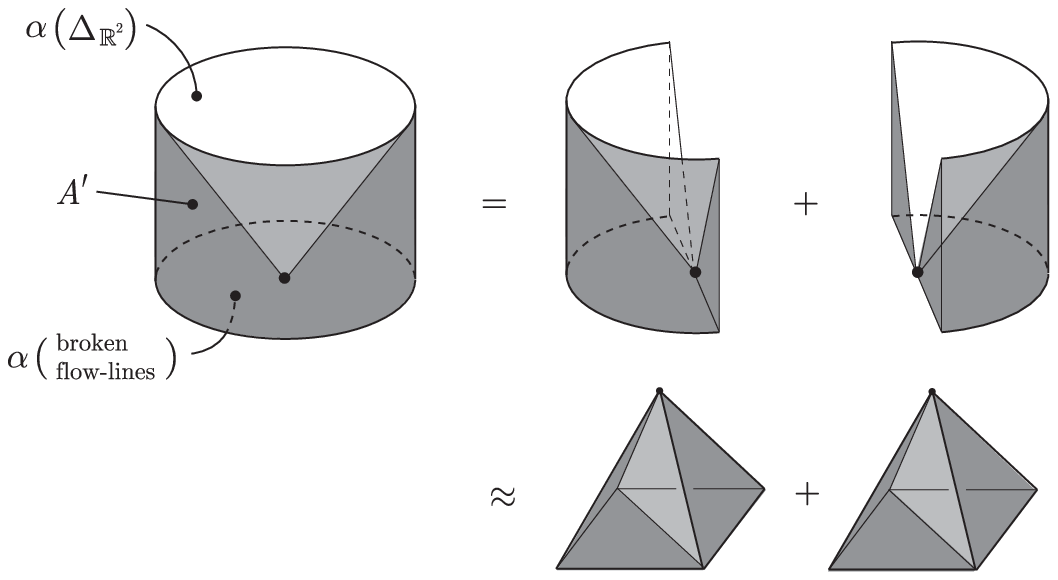}
\caption{}\label{fig:corner_diag}
\end{figure}
\end{proof}

\begin{Lem}\label{lem:bd_diag}
The boundary of the $\Q(t)$-chain $P(\wf)$ concentrates on the lift of $\Delta_M$.
\end{Lem}
\begin{proof}
By Proposition~\ref{prop:M2(f)}, strata $\bS_{r,0}(\wf;\rho)$ are glued together along the codimension 1 strata $\bS_{r,1}(\wf;\rho)$ and $\bT_{r,0}(\wf;\rho)$. Proposition~\ref{prop:M2(f)} shows that for each codimension 1 strata $T$ of $\bacalM_2(\wf)_\Z$, there are at most two codimension 0 strata having $T$ as a face. So it suffices to check that the orientations of the strata $\bS_{r,0}(\wf;\rho)$ are consistent at the codimension 1 strata $\bS_{r,1}(\wf;\rho)$ and $\bT_{r,0}(\wf;\rho)$.

The consistency at $\bS_{r,1}(\wf;\rho)$ has been proved in Lemma~\ref{lem:coori_consistent}. 

The consistency at $\bT_{r,0}(\wf;\rho)$: It suffices to check the consistency at the two basic relations considered in \S\ref{ss:def_M2}. 

{\bf Edge relation.} The proof is done by direct computations with the convention given in \S\ref{ss:coori_S}. For example, suppose $(x,z,y)\in\bS_{0,0}(\wf;\wW_i^{(j)},\wW_{i-1}^{(j)})\subset \wW_i^{(j)}\times \wL_i^{(j)}\times\wW_{i-1}^{(j)}$ is such that both $x$ and $y$ are close to $z$. Let $\wL_x,\wL_z,\wL_y$ be the level surfaces of $\wf$ including $x,z,y$ respectively. Let $dx_1,dx_2\in T_x^*\wL_x$ be the dual basis of an orthonormal basis of $T_x\wL_x$ such that $dx_1\wedge dx_2=o(\wL_x)_x$. Let $\{dz_1,dz_2\}\subset T_z^*\wL_z$ and $\{dy_1,dy_2\}\subset T_y^*\wL_y$ be the bases induced by the gradient flow for $-\widetilde{\xi}$ from $\{dx_1,dx_2\}$. Then by convention, we have 
\[ \begin{split}
&o^*(\bS_{0,0}(\wf;\wW_i^{(j)},\wW_{i-1}^{(j)}))_{(x,z,y)}\\
=&o^*(\bS_{0,0}(\wf;\wW_i^{(j)},\wL_i^{(j)}))_{(x,z)}\wedge o^*(\bS_{0,0}(\wf;\wL_i^{(j)},\wW_{i-1}^{(j)}))_{(z,y)}\\
=&(dx_1-dz_1)\wedge (dx_2-dz_2)
\wedge (dz_1-dy_1)\wedge (dz_2-dy_2)+O(\sqrt{d(x,z)^2+d(z,y)^2})\\
=&o(\wL_i^{(j)})_z\wedge (dx_1-dy_1)\wedge (dx_2-dy_2)+O(\sqrt{d(x,z)^2+d(z,y)^2}),
\end{split}\]
where $d(\cdot,\cdot)$ is the geodesic distance on $\wM$. Indeed, $\langle dx_j-dz_j, dx_k+dz_k\rangle =0 + O(d(x,z))$, $\langle dz_j-dy_j,dz_k+dy_k\rangle = 0+O(d(z,y))$, the exterior product of 
\[ o(S_{0,0}(\wf;\wW_i^{(j)},\wW_{i-1}^{(j)}))_{(x,z,y)}=(-d\wf)_y\wedge (-d\wf)_x\wedge (dx_1+dz_1+dy_1)\wedge (dx_2+dz_2+dy_2)\]
with $(dx_1-dz_1)\wedge (dx_2-dz_2)
\wedge (dz_1-dy_1)\wedge (dz_2-dy_2)$ is $9(-d\wf)_y\wedge (-d\wf)_x\wedge (dx_1\wedge dx_2)\wedge (dz_1\wedge dz_2)\wedge (dy_1\wedge dy_2)=9\,o(\wW_i^{(j)})_y\wedge o(\wL_i^{(j)})_z\wedge o(\wW_i^{(j)})_x$. 
On the other hand, if moreover $(x,y)\in\bS_{0,0}(\wf;\wW_i^{(j)},\wL_i^{(j)})\subset \wW_i^{(j)}\times \wL_i^{(j)}$, then 
\[ o^*(\bS_{0,0}(\wf;\wW_i^{(j)},\wL_i^{(j)}))_{(x,y)}
=(dx_1-dy_1)\wedge (dx_2-dy_2)+O(d(x,y)).\]
The two coorientations contribute to the sign in the $\Q(t)$-chain $P(\wf)$ in the same way. 

{\bf Vertex relation.} For the relation (1), let $(B_1,B_2,B_3)$ (resp. $(B_1,B_2',B_3)$) be the three successive cells that correspond to the left hand side (resp. right hand side) of (1). Let $\rho$ and $\rho'$ be the descending routes corresponding to the two sides of (1). The spaces $X_\rho$ and $X_{\rho'}$ are of the following forms.
\[ \begin{split}
  X_\rho=X_1\times \wL(B_1)\times X_2, \quad X_{\rho'}=X_1\times \wL(B_2')\times X_2.
\end{split}\]
The union of $\wL(B_1)$ and $\wL(B_2')$ is smooth. By convention for the orientation of level surface, we have
\[ o(\wL(B_1))_z=\iota(-\widetilde{\xi}_z)\,o(\wM)_z=o(\wL(B_2'))_z \]
for $z\in \wL(B_1)\cap \wL(B_2')$. This shows the consistency of the orientations on $X_\rho\cap X_{\rho'}$. The relations (2) $\sim$ (5) are similar to this case.

For the relation (6), let $(B_1,B_3)$ (resp. $(B_1,B_2',B_3)$) be the successive cells on the left hand side (resp. right hand side) of (6). It is enough to consider the case $\rho=(B_1,B_3)$ and $\rho'=(B_1,B_2',B_3)$. In this case, the spaces $X_\rho$ and $X_{\rho'}$ are as follows.
\[ \begin{split}
  X_\rho=\wW(B_1)\times \wW(B_3),\quad X_{\rho'}=\wW(B_1)\times \wL(B_2')\times \wW(B_3).
\end{split}\]
The consistency for $o^*(\bS_{0,0}(\wf;\rho))$ and $o^*(\bS_{0,0}(\wf;\rho'))$ is similar to that of the edge relation. The consistency for $o^*(\wcalD_\gamma)_y$ and $o^*(\wcalD_\gamma)_z\wedge o^*(\bS_{0,0}(\wf;\wL(B_2'),\wW(B_3)))_{(z,y)}$, $y\in \wcalD_\gamma\cap\wW(B_1)\cap \wW(B_3)$, $z\in \wcalD_\gamma\cap \wW(B_1)\cap \wL(B_2')\cap \wW(B_3)$, $\gamma\cap \wW(B_2')\neq\emptyset$, $\mathrm{ind}\,\gamma=1$, can be checked as follows. (The case $\mathrm{ind}\,\gamma=2$ is easier than this case.)
\[ \begin{split}
  &o^*(\wcalD_\gamma)_z\wedge o^*(\bS_{0,0}(\wf;\wL(B_2'),\wW(B_3)))_{(z,y)}\\
  =&\mu\, o(\wL_z)_z\wedge o^*(\wcalD_\gamma)_y+o^*(\wcalD_\gamma)_z\wedge o(\wL_y)_y
\end{split}\]
for some $\mu>0$. Then $(\mathrm{ev}_z\otimes 1)(o^*(\wcalD_\gamma)_z\wedge o^*(\bS_{0,0}(\wf;\wL(B_2'),\wW(B_3)))_{(z,y)})=\mu\, o^*(\wcalD_\gamma)_y$, where $\mathrm{ev}_z:\bigwedge^\bullet\, T_z^*\wL_z\to \R$, $\mathrm{ev}_z(\omega)=\omega(v)$, is the evaluation with the orientation framing $v\in \bigwedge^2\,T_z\wL_z$. This shows that the two coorientations induce consistent orientations. Hence the boundaries of $\bS_{r,0}(\wf;\rho)$ and $\bS_{r,0}(\wf;\rho')$ cancel with each other along their intersection. The relation (7) is the same as this case.
\end{proof}

\subsection{Blow-up along the diagonal}

\begin{Def}
Let $B\ell_{\bar{b}^{-1}(\widetilde{\Delta}_M)}(\bacalM_2(\wf)_\Z)$ be the blow-up of $\bacalM_2(\wf)_\Z$ along $\bar{b}^{-1}(\widetilde{\Delta}_{M})$, i.e., the union of the blow-ups of the smooth manifold strata. Let $Q(\wf)$ be the 4-dimensional $\Q(t)$-chain in $\bConf_2(\wM)_\Z$ represented by the natural map $B\ell_{\bar{b}^{-1}(\widetilde{\Delta}_M)}(\bacalM_2(\wf)_\Z)\to \bConf_2(\wM)_\Z$ induced by $\bar{b}$.
\end{Def}

We say that an AL-path $\gamma$ in $\wM$ is an {\it AL-cycle} if $\bar{b}(\gamma)\in\widetilde{\Delta}_M$. In other words, if the endpoints of $\gamma$ are $x$ and $y$, then $\gamma$ is an AL-cycle if moreover $\pi(x)=\pi(y)$. An AL-cycle $\gamma$ descends to a piecewise smooth map $\bar{\gamma}:S^1\to M$, which can be considered as a ``closed orbit'' in $M$. We will also call $\bar{\gamma}$ an AL-cycle. An AL-cycle has an orientation that is determined by the orientations of descending and ascending manifolds loci of $\widetilde{\xi}$. Then we define the sign $\ve(\gamma)\in\{-1,1\}$ and the period $p(\gamma)$ of $\gamma$ by the following equation
\[ [\bar{\gamma}]=\ve(\gamma)\,p(\gamma)[K] \]
in $H_1(M)$, where $K$ is a knot in $M$ such that $\langle[d\kappa],[K]\rangle=-1$ and $p(\gamma)$ is a positive integer. In other words, 
\[ p(\gamma)=|\langle[d\kappa],[\bar{\gamma}]\rangle|,\quad
\ve(\gamma)=-\frac{\langle[d\kappa],[\bar{\gamma}]\rangle}{|\langle[d\kappa],[\bar{\gamma}]\rangle|}. \]
Let $ST(\gamma)$ be the pullback $\bar{\gamma}^*ST(M)$, which can be considered as a piecewise smooth 3-dimensional chain in $\partial\bConf_2(\wM)_\Z$. We say that two AL-cycles $\gamma_1$ and $\gamma_2$ are {\it equivalent} if there is a degree 1 homeomorphism $g:S^1\to S^1$ such that $\bar{\gamma}_1\circ g=\bar{\gamma}_2$. The indices of vertical segments in an AL-cycle must be all equal since an AL-path is descending. We define the index $\mathrm{ind}\,\gamma$ of an AL-cycle $\gamma$ to be the index of a vertical segment in $\gamma$. 

Let $M_0=M\setminus\bigcup_{\gamma:\mathrm{critical\,locus}}\gamma$ and let $s_\xi:M_0\to ST(M_0)$ be the normalization $-\xi/\|\xi\|$ of the section $-\xi$. The closure $\overline{s_\xi(M_0)}$ in $ST(M)$ is a smooth manifold with boundary whose boundary is the disjoint union of circle bundles over the critical loci $\gamma$ of $\xi$, for a similar reason as \cite[Lemma~4.3]{Sh}. The fibers of the circle bundles are equators of the fibers of $ST(\gamma)$. Let $E^-_\gamma$ be the total space of the 2-disk bundle over $\gamma$ whose fibers are the lower hemispheres of the fibers of $ST(\gamma)$ which lie below the level surfaces of $\kappa$. Then $\partial\overline{s_\xi(M_0)}=\bigcup_\gamma \partial E_\gamma^-$ as sets. Let
\[ s_\xi^*(M)=\overline{s_\xi(M_0)}\cup \bigcup_\gamma E^-_\gamma\subset ST(M).\]
This is a 3-dimensional piecewise smooth manifold. We orient $s_\xi^*(M)$ by extending the natural orientation $(s_\xi^{-1})^*o(M)$ on $s_\xi(M_0)$ induced from the orientation $o(M)$ of $M$. The piecewise smooth projection $s_\xi^*(M)\to M$ is a homotopy equivalence and $s_\xi^*(M)$ is homotopic to $s_{\hat{\xi}}$. 

Let us fix an orientation of $ST(M)$ and $ST(\gamma)$. Recall that $B\ell_0(\R^3)$ can be identified with the closure in $S^2\times \R^3$ of the image of the section $s:\R^3\setminus\{0\}\to S^2\times (\R^3\setminus\{0\})$, $s(x)=(\frac{x}{\|x\|},x)$. Let $\mathrm{pr}_1:S^2\times (\R^3\setminus\{0\})\to S^2$ and $\mathrm{pr}_2:S^2\times(\R^3\setminus\{0\})\to \R^3\setminus\{0\}$ be the projections. Let $\omega$ be the closed 2-form on $\R^3\setminus\{0\}$ given by
\[ \omega(x)=\frac{1}{\|x\|^3}(x_1\,dx_2\wedge dx_3-x_2\,dx_1\wedge dx_3+x_3\,dx_1\wedge dx_2). \]
Then the pullback $\mathrm{pr}_1^*\omega$ agrees on $\mathrm{Im}\,s$ with $\mathrm{pr}_2^*\omega$. This shows that $\mathrm{pr}_1^*\omega$ can be smoothly extended over $B\ell_0(\R^3)$ by $\mathrm{pr}_2^*\omega$. Now we orient $ST(M)$ by the smooth extension of 
\[ \phi^*(\mathrm{pr}_1^*\omega\wedge o(\Delta_M)), \]
where $\phi$ is the trivialization in (\ref{eq:triv_normal}) and $\mathrm{pr}_1:\R^3\setminus\{0\}\times\Delta_M\to \R^3\setminus\{0\}$ is the projection. Here, the standard orientation $o(\Delta_M)$ of $\Delta_M$ is given as follows. If $T_x^*M$ is spanned by $e_1,e_2,e_3$ and if $o(M)_x=e_1\wedge e_2\wedge e_3$, then $T^*_{(x,x)}\Delta_M\subset T_x^*M\oplus T_x^*M$ is spanned by $\mathrm{pr}_1^*e_i+\mathrm{pr}_2^*e_i$, $i=1,2,3$, and we define $o(\Delta_M)_{(x,x)}=\bigwedge_{i=1}^3(\mathrm{pr}_1^*e_i+\mathrm{pr}_2^*e_i)$. It is easy to check that $\phi^*(\mathrm{pr}_1^*\omega\wedge o(\Delta_M))$ is equivalent to
\[ \iota(n)\,\mathrm{pr}_2^*o(M)\wedge \mathrm{pr}_1^*o(M), \]
where $n$ is a vector field on $N_{\Delta_M}\setminus\Delta_M$ that is outward normal with respect to $\Delta_M$. Similarly, we orient $ST(\gamma)$ for an AL-cycle $\gamma$ by
\[ \phi^*(\mathrm{pr}_1^*\omega\wedge o(\Delta_\gamma)), \]
where $o(\Delta_\gamma)=\mathrm{pr}_1^* o(\gamma)+\mathrm{pr}_2^* o(\gamma)$. 

For an AL-cycle $\gamma$, we denote by $\gamma^\irr$ the minimal AL-cycle such that $\gamma$ is equivalent to the iteration $(\gamma^\irr)^k$ for a positive integer $k$ and we call $\gamma^\irr$ the {\it irreducible} factor of $\gamma$. This is unique up to equivalence. 

\begin{Thm}\label{thm:1}
The boundary of the 4-dimensional $\Q(t)$-chain $Q(\wf)$ is given by
\[ \partial Q(\wf)=s_{\xi}^*(M)+\sum_\gamma (-1)^{\mathrm{ind}\,\gamma}\ve(\gamma)\,t^{p(\gamma)}\,ST(\gamma^\irr), \]
where the sum is taken over equivalence classes of AL-cycles in $\wM$.
\end{Thm}

\begin{Lem}\label{lem:bQ_cycle}
The face of $\partial Q(\wf)$ at an AL-cycle $\gamma$ without horizontal segments contributes as $(-1)^{\mathrm{ind}\,\gamma}t^{p(\gamma)}ST(\gamma^\irr)$.
\end{Lem}
\begin{proof}
Suppose that $\gamma$ is an AL-cycle from $x_0$ to $t^{i}x_0$, where $i$ is the period of $\gamma$. Let $\rho=(B(1),B(2),\ldots,B(N))$ be a descending route with one block such that for every $i$, $\wW(B(i))$ intersects $\gamma$ and $\wW(B(N))=t^{i}\wW(B(1))$. Then it follows that $X_\rho=\wW(B(1))\times\wW(B(N))=\wW(B(1))\times t^{i}\wW(B(1))\approx \wW(B(1))\times\wW(B(1))$.

First, we consider the case where $\mathrm{ind}\,\gamma=1$. Let $x_0$, $U_{x_0}$, $\widetilde{A}$, $\widetilde{D}$ be as in the first half ($\mathrm{ind}\,\gamma=1$ case) of the proof of Lemma~\ref{lem:coori_consistent}. Then
\[ \begin{split}
  \bS_{1,0}(\wf;\rho)\cap (U_{x_0}\times t^{i}U_{x_0})
  =\widetilde{A}\times t^{i}\widetilde{D}
\end{split} \]
if $i\geq 1$. By convention of \S\ref{ss:ori_DA}, the coorientation of $\widetilde{A}\times \widetilde{D}$ is given by
\[ o^*(\widetilde{A})_x\wedge o^*(\widetilde{D})_{x'}=dx_2\wedge dx_3'. \]
By Remark~\ref{rem:oMM}, this gives 
\[ o(\widetilde{A}\times\widetilde{D})_{(x,x')}
=dx_1\wedge dx_3\wedge dx_1'\wedge dx_2'. \]
The outward normal vector field to $\Delta_\gamma$ in $(\widetilde{A}\times\widetilde{D})\setminus\Delta_\gamma$ is given by $(-x_1',0,x_3,x_1',x_2',0)$. So the induced orientation at the boundary of the blow-up is given by the formula
\begin{equation}\label{eq:o(STgamma)}
 \begin{split}
  &\iota(-x_1',0,x_3,x_1',x_2',0)\,dx_1\wedge dx_3\wedge dx_1'\wedge dx_2'\\
&=-(dx_1+dx_1')\wedge (x_1'\,dx_2'\wedge dx_3 -x_2'\,dx_1'\wedge dx_3 +x_3\,dx_1'\wedge dx_2').
\end{split}
\end{equation}
Let $\phi:(\widetilde{A}\times \widetilde{D})\setminus\Delta_{U_{x_0}}\to \R^3\setminus\{0\}$ be the map given by
\[ \phi(x_1,0,x_3,x_1',x_2',0)
=(x_1'-x_1,x_2',-x_3). \]
Then 
\[ \phi^*\omega(x,x')=\frac{1}{\|\phi(x,x')\|^3}
\Bigl(
-(x_1'-x_1)\,dx_2'\wedge dx_3 + x_2'\, (dx_1'-dx_1)\wedge dx_3 
-x_3\,(dx_1'-dx_1)\wedge dx_2'
\Bigr).\]
We consider the induced 2-form on the fiber $F$ of $N_{\Delta_M}\setminus \Delta_M$. So we may impose the relation $x_1=-x_1'$. Then the form $\phi^*\omega$ induces
\[ \phi^*\omega|_F(x,x')=\frac{1}{\|\phi(x,x')\|^3}
(
-2x_1'\,dx_2'\wedge dx_3 + 2x_2'\, dx_1'\wedge dx_3 
-2x_3\,dx_1'\wedge dx_2'
).
\]
Since the orientation of $\Delta_\gamma$ is given by $-(dx_1'+dx_1)$, the orientation of $ST(\gamma)$ is given by
\[ \frac{2}{\|\phi(x,x')\|^3}(dx_1'+dx_1)\wedge(
x_1'\,dx_2'\wedge dx_3 -x_2'\, dx_1'\wedge dx_3 
+x_3\,dx_1'\wedge dx_2'
).
\]
This is opposite to (\ref{eq:o(STgamma)}). Hence the smooth extensions of these forms to the boundary give opposite orientations. 

Next, we consider the case where $\mathrm{ind}\,\gamma=2$. Let $x_0$, $U_{x_0}$, $\widetilde{A}$, $\widetilde{D}$ be as in the last half (case $\mathrm{ind}\,\gamma=2$) of the proof of Lemma~\ref{lem:coori_consistent}. Then
\[ \begin{split}
  \bS_{1,0}(\wf;\rho)\cap (U_{x_0}\times t^{i}U_{x_0})
  =\widetilde{A}\times t^{i}\widetilde{D}
\end{split} \]
if $i\geq 1$. By convention of \S\ref{ss:ori_DA}, the coorientation of $\widetilde{A}\times \widetilde{D}$ is given by
\[ o^*(\widetilde{A})_x\wedge o^*(\widetilde{D})_{x'}=-dx_2\wedge dx_3. \]
By Remark~\ref{rem:oMM}, this gives 
\[ o(\widetilde{A}\times\widetilde{D})_{(x,x')}
=dx_1\wedge dx_1'\wedge dx_2'\wedge dx_3'. \]
The induced orientation at the boundary is given by
\begin{equation}\label{eq:o(STgamma2)}
 \begin{split}
  &\iota(-x_1',0,0,x_1',x_2',x_3')\,dx_1\wedge dx_1'\wedge dx_2'\wedge dx_3'\\
&=-(dx_1+dx_1')\wedge (x_1'\,dx_2'\wedge dx_3' -x_2'\,dx_1'\wedge dx_3' +x_3'\,dx_1'\wedge dx_2')
\end{split}
\end{equation}
Let $\phi:(\widetilde{A}\times \widetilde{D})\setminus\Delta_{U_{x_0}}\to \R^3\setminus\{0\}$ be the map given by
\[ \phi(x_1,0,0,x_1',x_2',x_3')
=(x_1'-x_1,x_2',x_3'). \]
Then 
\[ \phi^*\omega(x,x')=\frac{1}{\|\phi(x,x')\|^3}
\Bigl(
(x_1'-x_1)\,dx_2'\wedge dx_3' - x_2'\, (dx_1'-dx_1)\wedge dx_3' 
+x_3'\,(dx_1'-dx_1)\wedge dx_2'
\Bigr).\]
We consider the induced 2-form on the fiber $F$ of $N_{\Delta_M}$. So we may impose the relation $x_1=-x_1'$. Then the form $\phi^*\omega$ induces
\[ \phi^*\omega|_F(x,x')=\frac{1}{\|\phi(x,x')\|^3}
\Bigl(
2x_1'\,dx_2'\wedge dx_3' - 2x_2'\, dx_1'\wedge dx_3' 
+2x_3'\,dx_1'\wedge dx_2'
\Bigr).
\]
Since the orientation of $\Delta_\gamma$ is given by $-(dx_1'+dx_1)$, the orientation of $ST(\gamma)$ is given by
\[ -\frac{2}{\|\phi(x,x')\|^3}
(dx_1'+dx_1)\wedge(
x_1'\,dx_2'\wedge dx_3' - x_2'\, dx_1'\wedge dx_3' 
+x_3'\,dx_1'\wedge dx_2'
).\]
This is equivalent to (\ref{eq:o(STgamma2)}). Hence the smooth extensions of these forms to the boundary give equivalent orientations. 

The case where $\mathrm{ind}\,\gamma=0$ is the same as the case where $\mathrm{ind}\,\gamma=2$. 

The reason that we must consider the irreducible factor $\gamma^\irr$ would be clear if one considers a $p$-fold covering $c:S^1\to S^1$. A choice of base point $v\in S^1$ gives a lift $\widetilde{c}_v:[0,1]\to \R$ of $c$ in the universal covering $\pi:\R\to S^1$. The endpoints of $\widetilde{c}_v$ defines a point $v_0\times_\Z v_1$ in $\R\times_\Z\R$, which is defined analogously to $\wM\times_\Z\wM$. Conversely, the set of points $v_0\times_\Z v_1\in \R\times_\Z\R$ such that $\pi(v_0)=\pi(v_1)$ and $v_1=t^p v_0(=v_0-p)$ is of the form $t^p\, \widetilde{c}_v([0,\frac{1}{p}])$, which can be written as $t^p\,c^\irr$. This explains the reason for the term $t^{p(\gamma)}ST(\gamma^\irr)$.  
\end{proof}

\begin{Lem}\label{lem:bQ_diag}
The face of $\partial Q(\wf)$ at $ST(M)$ contributes as $s_\xi^*(M)$. 
\end{Lem}
\begin{proof}
It suffices to check that the induced orientation on $\partial Q(\wf)\cap ST(M)$ is equivalent to the standard one on $s_\xi^*(M)$. So we consider a pair $(x,x')\in \wM\times \wM$ of points both not close to any critical loci but close to each other. By convention of \S\ref{ss:coori_S}, the orientation of $Q(\wf)$ at $(x,x')$ is given by
\[ o(Q(\wf))_{(x,x')}=(-d\wf)_{x'}\wedge o(\Delta_{\wM})_{(x,x)}+O(d(x,x')), \]
where $d(x,x')$ is the geodesic distance on $\wM$. 
The outward normal vector field in $Q(\wf)$ near $\bar{b}^{-1}(\Delta_{\wM})$ is given by $-\widetilde{\xi}_{x'}$. Thus the induced orientation on the boundary is
\[ \iota(-\widetilde{\xi}_{x'})\,(-d\wf)_{x'}\wedge o(\Delta_{\wM})_{(x,x)}=o(\Delta_{\wM})_{(x,x)}. \]
This is equivalent to the orientation of $s_\xi^*(M)$. 
\end{proof}

\begin{proof}[Proof of Theorem~\ref{thm:1}]
By Lemma~\ref{lem:bd_diag}, the boundary of $Q(\wf)$ concentrates on $\partial \bConf_2(\wM)_\Z$. The boundary of $\bConf_2(\wM)_\Z$ consists of faces made by the blow-up along $\bar{b}^{-1}(\widetilde{\Delta}_M)$. By Lemma~\ref{lem:bQ_diag}, the face made by the blow-up along $\bar{b}^{-1}(\Delta_{\wM})$ contributes as $s_\xi^*(M)$. The other faces correspond to the tangent sphere bundle over AL-cycles. Then Lemma~\ref{lem:bQ_cycle} finishes the proof.
\end{proof}

\subsection{Relation with the Lefschetz zeta function}

We shall see that the homology class of the term $\sum_\gamma(-1)^{\mathrm{ind}\,\gamma}\ve(\gamma)t^{p(\gamma)}ST(\gamma^\irr)$ in the formula of Theorem~\ref{thm:main} can be rewritten in terms of the Lefschetz zeta function.

\begin{Prop}\label{prop:gamma_zeta}
Consider $M$ as the mapping torus of an orientation preserving diffeomorphism $\varphi:\Sigma\to \Sigma$, where $\Sigma=\kappa^{-1}(0)$. Then the following identity holds.
\begin{equation}\label{eq:gamma_zeta1}
 \sum_\gamma(-1)^{\mathrm{ind}\,\gamma}\ve(\gamma)\,p(\gamma^\irr)\,t^{p(\gamma)}
=\frac{t\zeta'_\varphi}{\zeta_\varphi}, 
\end{equation}
where the sum is taken over equivalence classes of AL-cycles in $\wM$, or equivalently,
\[ \exp\left(\sum_\gamma(-1)^{\mathrm{ind}\,\gamma}\frac{\ve(\gamma)\,p(\gamma^\irr)}{p(\gamma)}t^{p(\gamma)}\right)=\zeta_\varphi. \]
\end{Prop}
\begin{proof}
Here, we assume that all the homology groups are considered with coefficients in $\Q$. The restriction of the fiberwise gradient $\xi$ to $\Sigma$ defines a handle filtration $\emptyset=\Sigma^{(-1)}\subset\Sigma^{(0)}\subset \Sigma^{(1)}\subset \Sigma^{(2)}=\Sigma$. Put $C_i(\Sigma)=H_i(\Sigma^{(i)},\Sigma^{(i-1)})$. Then $\varphi$ induces endomorphisms $\varphi_{\sharp i}:C_i(\Sigma)\to C_i(\Sigma)$ and $\varphi_{*i}:H_i(\Sigma)\to H_i(\Sigma)$. Note that $\varphi_{\sharp i}$ is uniquely determined for $\xi$ because $\varphi_{\sharp 0}$ and $\varphi_{\sharp 2}$ are induced by the permutation given by the graphic. Then $\varphi_{\sharp 1}$ can be seen as the induced map on $H_1$ of the corresponding base pointed homotopy equivalence $\bigvee S^1\to \bigvee S^1$, which is uniquely determined. (See \cite[Ch.~9]{Pa2}.)

By (\ref{eq:dlog_zeta}), the right hand side of (\ref{eq:gamma_zeta1}) can be rewritten as
\[ \frac{t\zeta'_\varphi}{\zeta_\varphi}
=\sum_{i=0}^2(-1)^i\mathrm{Tr}\frac{t\varphi_{*i}}{1-t\varphi_{*i}}
=\sum_{i=0}^2(-1)^i\mathrm{Tr}\frac{t\varphi_{\sharp i}}{1-t\varphi_{\sharp i}}
=\sum_{i=0}^2(-1)^i\sum_{k=1}^\infty t^k\mathrm{Tr}\,\varphi_{\sharp i}^k.
\]
Hence it suffices to check the identity
\begin{equation}\label{eq:sum_ep_tr}
 \sum_{{\gamma}\atop{\mathrm{ind}\,\gamma=i}}\ve(\gamma)\,p(\gamma^\irr)\,t^{p(\gamma)}
=\sum_{k=1}^\infty t^k\mathrm{Tr}\,\varphi_{\sharp i}^k. 
\end{equation}
If $i=0$ or $2$, then $\ve(\gamma)=1$ for any AL-cycle $\gamma$ and $\varphi_{\sharp i}$ is given by a permutation matrix since an AL-cycle having a vertical segment of index 0 or 2 can not have horizontal segments. Then the identity (\ref{eq:sum_ep_tr}) is immediate. The case $i=1$ is more complicated since there may be AL-cycles that pass through $1/1$-intersections. Let $p_1,p_2,\ldots,p_N$ be the critical points of $\xi|_\Sigma$. We identify $C_1(\Sigma)$ with the free abelian group generated by $\{p_1,\ldots,p_N\}$. We define an endomorphism $\Theta(\xi):C_1(\Sigma)\to C_1(\Sigma)$ by
\[ \Theta(\xi)(p_i)=\sum_{j=1}^N N(p_i,p_j)\,p_j, \]
where $N(p_i,p_j)\in\Z$ is the count of AL-paths from $p_i$ to $tp_j$ counted with orientations as in \S\ref{ss:sign_AL}. Then by definition of $\Theta(\xi)$, we have
\[ \sum_{{\gamma}\atop{\mathrm{ind}\,\gamma=1}}\ve(\gamma)\,p(\gamma^\irr)\,t^{p(\gamma)}
=\sum_{k=1}^\infty t^k\mathrm{Tr}\,\Theta(\xi)^k. \]
Indeed, the right hand side counts AL-cycles of period $k$ with a base point on $\Sigma$. On the other hand, the sum in the left hand side is over AL-cycles without base point. The multiplicity of an AL-cycle $\gamma$ without base point in the right hand side is exactly $p(\gamma^\irr)$. This proves the identity above. Therefore, it suffices to check that $\varphi_{\sharp 1}$ agrees with $\Theta(\xi)$. This is proved in Lemma~\ref{lem:phi=Theta} below.
\end{proof}

\begin{Lem}\label{lem:phi=Theta}
$\varphi_{\sharp 1}=\Theta(\xi)$.
\end{Lem}
\begin{proof}
Recall the decomposition $S^1=\bigcup_{j=1}^{2r}I_j$ considered in \S\ref{ss:M2(f)}. The both sides of Lemma~\ref{lem:phi=Theta} can be decomposed as the compositions of corresponding morphisms $\varphi_{\sharp 1}^{I_j}$ and $\Theta^{I_j}(\xi)$ on $I_j$ from top to bottom. We check that $\varphi_{\sharp 1}^{I_j}$ and $\Theta^{I_j}(\xi)$ coincide for each $j$. 

If $j$ is odd, then there are no $1/1$-intersections in $\kappa^{-1}(I_j)$. $\varphi_{\sharp 1}^{I_j}$ and $\Theta^{I_j}(\xi)$ are given by the same permutation on the set of the critical points, totally ordered by the fiberwise Morse function $f$. So we need only to consider the case that $j$ is even. By slicing $I_j$ further into small intervals $I_{jk}=[a_{jk},b_{jk}]$ each containing just one $1/1$-intersection, it suffices to check Lemma~\ref{lem:phi=Theta} for the case when there is only one $1/1$-intersection and there is no level exchange bifurcations in $I_{jk}$. It can be seen by using the Morse lemma that a $1/1$-intersection corresponds to a slide of a 1-handle over another 1-handle, as is well known (e.g. \cite[Theorem~7.6]{Mi}). There are two possibilities for a handle-sliding, as shown in Figure~\ref{fig:1-handle-slide}. Moreover, there are four possibilities for the coorientations of the two descending manifold loci that are the cores of the two 1-handles. 
\begin{figure}
\fig{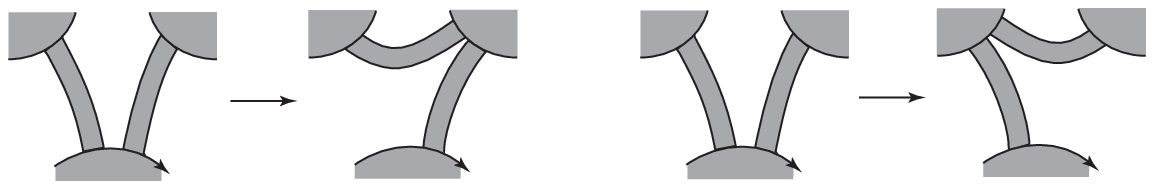}
\caption{}\label{fig:1-handle-slide}
\end{figure}

We consider the first case in Figure~\ref{fig:1-handle-slide}. The second case is similar to the first one. Let $p,q$ be critical loci of $\xi|_{M_{I_{jk}}}$ and suppose that $\wcalD_p$ slides over $\wcalD_q$ in $\xi|_{M_{I_{jk}}}$, as $s$ decreases from $b_{jk}$ to $a_{jk}$. Let $p_0\in p,q_0\in q$ be the endpoints of the flow-line of $\xi|_{M_{I_{jk}}}$ from $p$ to $q$ and take a local coordinate $(x_1,x_2,x_3)$ around $q_0$ such that
\begin{itemize}
\item the $x_2x_3$-plane agrees with the level surface $\Sigma_{q_0}$ of $\kappa$ at $q_0$,
\item $x_1$-axis points the upward direction,
\item the $x_1x_2$-plane agrees with $\wcalD_q$,
\item the $x_1x_3$-plane agrees with $\wcalA_q$.
\end{itemize}
We only consider one special case about the orientations out of the four since the other cases can be checked by the same argument as the special case. So we assume the following, applying the convention in \S\ref{ss:ori_DA}.
\[ \left\{\begin{array}{ll}
o(\calD_{q_0})_x=dx_2,&o(\wcalD_q)_x=-dx_1\wedge dx_2,\\
o(\calD_{p_0})_x=-dx_3,&o(\wcalD_p)_x=-(a\,dx_1+dx_2)\wedge (-dx_3)
\end{array}\right. \]
for a real number $a>0$. See Figure~\ref{fig:slide_local}. This gives
\[ \left\{\begin{array}{l}
o^*(\wcalD_q)_x=-dx_3,\\
o^*(\wcalD_p)_x=dx_1-a\,dx_2,
\end{array}\right. \]
and
\[ o^*(\wcalD_p)_x\wedge o^*(\wcalA_q)_x=(dx_1-a\,dx_2)\wedge (-dx_2)=-dx_1\wedge dx_2=-o(\wL)_x, \]
where $\wL$ is the level surface locus of $f$ in $\kappa^{-1}(I_{jk})$ including $x$. From this and \S\ref{ss:sign_AL}, we have
\[ \Theta^{I_{jk}}(\xi)(\gamma)=\left\{\begin{array}{ll}
p-q & \mbox{if $\gamma=p$}\\
\gamma & \mbox{other critical locus}
\end{array}\right. \]
This agrees with the homological action $\varphi_{\sharp 1}^{I_{jk}}$ of the homotopy equivalence $\varphi^{I_{jk}}:\Sigma^{(1)}(b_{jk})/\Sigma^{(0)}(b_{jk})\to \Sigma^{(1)}(a_{jk})/\Sigma^{(0)}(a_{jk})$. More precisely, for the 1-handles $h_p,h_q,h_p',h_q'$ in Figure~\ref{fig:slide_local}, we have
\[ \varphi_{\sharp 1}^{I_{jk}}([h_p])=[h_p']-[h_q']\in H_1(\Sigma^{(1)}(a_{jk}),\Sigma^{(0)}(a_{jk})). \]
\begin{figure}
\fig{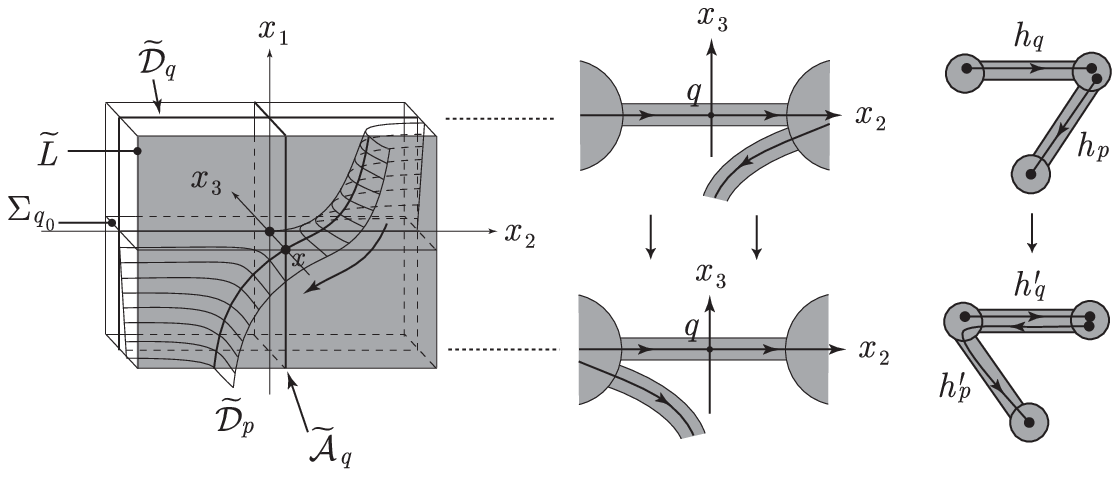}
\caption{}\label{fig:slide_local}
\end{figure}
\end{proof}

\begin{Cor}\label{cor:1}
Let $K$ be a knot in $M$ such that $\langle[d\kappa],[K]\rangle=-1$. Then
\[ [\partial Q(\wf)]=[s_{\hat{\xi}}(M)]+\frac{t\zeta'_\varphi}{\zeta_\varphi}[ST(K)] \]
in $H_3(\partial\bConf_2(\wM)_\Z)\otimes_\Lambda\Q(t)$.
\end{Cor}
\begin{proof}
This follows from Theorem~\ref{thm:1}, Proposition~\ref{prop:gamma_zeta} and $[ST(\gamma^\irr)]=p(\gamma^\irr)[ST(K)]$.
\end{proof}

\begin{appendix}

\mysection{Homology of a 3-manifold fibered over $S^1$}{s:appendix_1}

\begin{Lem}\label{lem:homology-mtorus}
Let $M$ be a closed, connected, oriented 3-manifold that is the mapping torus of a diffeomorphism $\varphi:\Sigma\to \Sigma$ of a closed, connected, oriented surface $\Sigma$. Then $b_1(M)=1$ if and only if $\varphi_{*1}:H_1(\Sigma;\Q)\to H_1(\Sigma;\Q)$ does not have eigenvalue $1$. More precisely, $b_1(M)=\dim{V_1}+1$, where $V_1\subset H_1(\Sigma;\Q)$ is the eigenspace of $\varphi_{*1}$ associated with eigenvalue $1$.
\end{Lem}
\begin{proof}
Let $\kappa:M\to S^1$ be the projection of the fibration. By \cite[Corollary~9.14]{DK}, there is an exact sequence
\[ 0\to H_0(S^1;H_1(\Sigma;\Q))\to H_1(M;\Q)\stackrel{\kappa_*}{\to} H_1(S^1;\Q)\to 0, \]
where the action of $\pi_1S^1$ on the local coefficient $H_1(\Sigma;\Q)$ is given by $\varphi_{*1}$. Hence it follows that the map $\kappa_*$ is an isomorphism if and only if the twisted homology $H_0(S^1;H_1(\Sigma;\Q))$ vanishes. Moreover, this condition is equivalent to the condition that the boundary homomorphism $\partial_1:C_1(S^1;H_1(\Sigma;\Q))\to C_0(S^1;H_1(\Sigma;\Q))$ is surjective. Decompose $S^1$ into one 0-cell $v$ and one 1-cell $c$ and consider the groups $C_0(S^1;H_1(\Sigma;\Q))$ and $C_1(S^1;H_1(\Sigma;\Q))$ as the cellular chain complex for this decomposition. For $a_c\in H_1(\Sigma_v;\Q)$, $\Sigma_v=\kappa^{-1}(v)$, we have
\[ \partial_1(a_c c)=\varphi_{*1}(a_c)v-a_cv=(\varphi_{*1}-1)(a_c)v. \]
Now the surjectivity (or equivalently, bijectivity) of $\partial_1$ is equivalent to $\det(\varphi_{*1}-1)\neq 0$. This completes the proof.
\end{proof}

In \cite{Ri}, it is shown that there are many closed 3-manifolds $M$ fibered over $S^1$ with $b_1(M)=1$. 

\mysection{Some facts on smooth manifolds with corners}{s:mfd_corners}

We follow the convention in \cite[Appendix]{BT} for manifolds with corners, smooth maps between them and their transversality. We write down some necessary terms from \cite[Appendix]{BT}, some of which are specialized than those in \cite[Appendix]{BT}.
\begin{Def}\label{def:mfd_corners}
\begin{enumerate}
\item A map between manifolds with corners is {\it smooth} if it has a local extension, at any point of the domain, to a smooth map from a manifold without boundary, as usual. 
\item Let $Y,Z$ be smooth manifolds with corners, and let $f:Y\to Z$ be a bijective smooth map. This map is a {\it diffeomorphism} if both $f$ and $f^{-1}$ are smooth.
\item Let $Y,Z$ be smooth manifolds with corners, and let $f:Y\to Z$ be a smooth map. This map is {\it strata preserving} if the inverse image by $f$ of a connected component $S$ of a stratum of $Z$ is a union of connected components of strata of $Y$. 
\item Let $X,Y$ be smooth manifolds with corners and $Z$ be a smooth manifold without boundary. Let $f:X\to Z$ and $g:Y\to Z$ be smooth maps. Say that $f$ and $g$ are {\it (strata) transversal} when the following is true: Let $U$ and $V$ be connected components in stratums of $X$ and $Y$ respectively. Then $f:U\to S$ and $g:V\to S$ are transversal.
\end{enumerate}
\end{Def}
We use the following proposition, which is a corollary of \cite[Proposition~A.5]{BT}.
\begin{Prop}\label{prop:BT}
Let $X,Y$ be smooth manifolds with corners and $Z$ be a smooth manifold without boundary. Let $f:X\to Z$ and $g:Y\to Z$ be smooth maps that are transversal. Then the fiber product
\[ X\times_Z Y=\{(x,y);f(x)=g(y)\}\subset X\times Y\]
is a smooth manifold with corners, whose strata have the form $U\times_Z V$ where $U\subset X$ and $V\subset Y$ are strata.
\end{Prop}

If $f,g$ are inclusions then $X\times_Z Y=(X\times Y)\cap \Delta_Z=\Delta_{X\cap Y}$, which is canonically diffeomorphic to $X\cap Y$. Thus we obtain the following corollary.

\begin{Cor}\label{cor:BT2}
Let $X,Y$ be smooth manifolds with corners that are submanifolds of a smooth manifold $Z$ without boundary. Suppose that the inclusions $X\to Z$ and $Y\to Z$ are transversal. Then the intersection $X\cap Y$ is a smooth manifold with corners, whose strata have the form $U\cap V$ where $U\subset X$ and $V\subset Y$ are strata.
\end{Cor}

\section{\bf Blow-up}\label{s:blow-up}

\subsection{Blow-up of $\R^i$ along the origin}

Let $\wgamma^1(\R^i)$ denote the total space of the tautological oriented half-line ($[0,\infty)$) bundle over the oriented Grassmannian $\widetilde{G}_1(\R^i)\cong S^{i-1}$. Namely, 
\[ \wgamma^1(\R^i)=\{(x,y)\in S^{i-1}\times \R^i; \exists t\in[0,\infty), y=tx\}. \]
Then the tautological bundle is trivial and that $\wgamma^1(\R^i)$ is diffeomorphic to $S^{i-1}\times [0,\infty)$. We put 
\[ B\ell_0(\R^i)=\wgamma^1(\R^i) \]
and call $B\ell_0(\R^i)$ the {\it blow-up} of $\R^i$ along $0$. Let $\pi:\wgamma^1(\R^i)\to \R^i$ be the map defined by $\pi=\mathrm{pr}_2\circ \varphi$ in the following commutative diagram:
\[ \xymatrix{
	\wgamma^1(\R^i) \ar[r]^-{\varphi} \ar[rd]_-{\pi}&
		S^{i-1}\times \R^i  \ar[d]^-{\mathrm{pr}_2} \ar[r]^-{\mathrm{pr}_1} & S^{i-1}\\
	& \R^i &
	}\]
where $\varphi:\wgamma^1(\R^i)\to S^{i-1}\times\R^i$ is the inclusion. We call $\pi$ the {\it projection} of the blow-up. Here, $\pi^{-1}(0)=\partial\wgamma^1(\R^i)$ is the image of the zero section of the tautological bundle $\mathrm{pr}_1\circ \varphi:\wgamma^1(\R^i)\to S^{i-1}$ and is diffeomorphic to $S^{i-1}$. 
\begin{Lem}\label{lem:bl_extention}\begin{enumerate}
\item The restriction of $\pi$ to the complement of $\pi^{-1}(0)=\partial\wgamma^1(\R^i)$ is a diffeomorphism onto $\R^i\setminus\{0\}$. 
\item The restriction of $\varphi$ to the complement of $\pi^{-1}(0)$ has the image in $S^{i-1}\times \R^i$ whose closure agrees with the full image of $\varphi$ from $\wgamma^1(\R^i)$. 
\end{enumerate}
\end{Lem}
\subsection{Blow-up along a submanifold}

When $d>i\geq 0$, we put
\[ B\ell_{\R^i}(\R^d)=\wgamma^1(\R^i)\times \R^{d-i} \]
(the blow-up of $\R^d$ along $\R^i$) and define the projection $\varpi:B\ell_{\R^i}(\R^d)\to \R^d$ by $\pi\times\mathrm{id}_{\R^{d-i}}$. This can be straightforwardly extended to the blow-up $B\ell_X(Y)$ of a manifold $Y$ along a submanifold $X$ having oriented normal bundle, by replacing the normal bundle with the associated $\wgamma^1(\R^d)$-bundle over $X$. 

\end{appendix}

\section*{\bf Acknowledgments.}
This work was started during my stay at the Institut Fourier of the University of Grenoble. I am grateful to the Institut for hospitality. Especially, I would like to express my sincere gratitude to Professor Christine Lescop for kindly helping me understand her work and for helpful conversations. I would like to thank Professor Osamu Saeki for helpful comments. I have been supported by JSPS Grant-in-Aid for Young Scientists (B) 70467447 and by the JSPS International Training Program organized by Department of Mathematics, Hokkaido University.

\end{document}